\documentclass[review,hidelinks,onefignum,onetabnum]{siamart251216}

\usepackage{lipsum}
\usepackage{amsfonts}
\usepackage{graphicx}
\usepackage{epstopdf}
\usepackage{algorithmic}

\usepackage{tikz}
\usepackage{subcaption}
\usepackage{bm}

\newcommand{\Vx}{\mathbf{x}}

\newcommand{\Vb}{\mathbf{b}}
\newcommand{\Vw}{\mathbf{w}}

\newcommand{\Vz}{\mathbf{z}}

\ifpdf
  \DeclareGraphicsExtensions{.eps,.pdf,.png,.jpg}
\else
  \DeclareGraphicsExtensions{.eps}
\fi

\newsiamremark{remark}{Remark}
\newsiamremark{hypothesis}{Hypothesis}
\crefname{hypothesis}{Hypothesis}{Hypotheses}
\newsiamthm{claim}{Claim}
\newsiamremark{fact}{Fact}
\crefname{fact}{Fact}{Facts}

\headers{Neural network with adaptive test spaces}{T. Udomworarat, I. Brevis, K. G.  van der Zee, and S. Rojas}

\title{Neural network approximation in discrete dual norms with adaptive test spaces}

\author{Tanakorn Udomworarat$^{\mathsection,}$\thanks{School of Mathematical Sciences, University of Nottingham, United Kingdom
(\email{tanakorn.udomworarat@nottingham.ac.uk},\email{ Ignacio.Brevis@nottingham.ac.uk},\email{ KG.vanderZee@nottingham.ac.uk}).}
\and Ignacio Brevis\footnotemark[1]
\and Kristoffer G. van der Zee\footnotemark[1]
\and Sergio Rojas\thanks{School of Mathematics, Monash University, Australia 
  (\email{ Sergio.rojas@monash.edu}).}}
{
\footnotetext{\hspace{-0.8em}$^\mathsection$Corresponding author.}}

\usepackage{amsopn}

\ifpdf
\hypersetup{
  pdftitle={An Example Article},
  pdfauthor={D. Doe, P. T. Frank, and J. E. Smith}
}
\fi

\begin{document}

\maketitle

\begin{abstract}
In robust variational physics-informed neural networks (RVPINNs), the loss function is formulated in terms of the Riesz representative of the variational residual within a discrete test space. This approach guarantees that the loss function is robust with respect to the true error in the energy norm up to a remainder term that depends on both the neural network approximation and the discrete space configuration. However, in problems with localized singularities, steep gradients, or interface layers, a fixed coarse test space may fail to resolve the continuous Riesz representative of the residual during training.
Although this can be avoided by using a sufficiently fine test space from the start, doing so may be computationally inefficient. We therefore propose an adaptive algorithm that enriches the test space only where the error between the discrete and continuous Riesz representatives is pronounced. We establish theoretical adaptive strategies within the RVPINN framework and derive their error bounds. Furthermore, we propose a computable refinement indicator and prove that, under the saturation assumption, it serves as a reliable and efficient error estimator for the non-computable discrepancy between the discrete and continuous Riesz representatives. Finally, we propose a practical adaptive algorithm and demonstrate its effectiveness through numerical experiments on elliptic Dirichlet problems.
\end{abstract}



\section{Introduction}
Deep learning methodologies have been widely applied to various areas of scientific computing over the past decade (see, for example,  \cite{aldirany2024multi,brunton2020machine,libbrecht2015machine,sirignano2018dgm,udomworarat2026neural,yu2018deep}). 
One popular application in this area is the use of neural networks to approximate solutions of partial differential equations (PDEs). This line of research gained significant attention with the introduction of physics-informed neural networks (PINNs) \cite{raissi2019physics} in 2019.  
The idea of PINNs is to generate training data by incorporating physical laws expressed as PDEs.
The neural network is trained using this data, which is on a mesh-free set of collocation points sampled from the physical domain and its boundary. PINNs have been successfully applied to a broad range of problems, including fluid dynamics \cite{cai2021physics,mao2020physics}, wave propagation \cite{rasht2022physics}, and inverse problems \cite{chen2020physics}. 
However, a caveat of this method is that it employs a strong-form residual to approximate the solution, whereas for many PDEs, the solution is only meaningful in a variational sense.
In addition, the mesh-free nature of PINNs is not always an advantage, as approximating solutions in complex scenarios, such as singular solutions or with steep gradients, can lead to convergence issues due to improper balancing of loss terms~\cite{fuhrer2025posteriori}, making training too expensive or impractical.

As an alternative to the strong-form residual, variational PINNs (VPINNs) \cite{kharazmi2019variational} and $hp$-VPINNs \cite{kharazmi2021hp} were introduced to solve PDEs in their weak forms.  
While these approaches align with weak solutions, they depend on the choice of basis functions and, consequently, their loss functions may not be robust with respect to the true error. 
Subsequently, robust VPINNs (RVPINNs) \cite{rojas2024robust} were proposed to overcome
\pagebreak these limitations.  
In this framework, the loss function is defined through a discrete Riesz representative of the residual. This construction ensures that the loss function acts as a reliable and efficient estimator 
for the true error up to a remainder term, which depends on the neural network approximation and the discretization of the test space. When standard FEM functions are used as test functions, the underlying mesh provides a natural framework for adaptive refinement that improves approximation quality. 

An analysis of the quadrature rules and piecewise-polynomial degrees of the test space for VPINNs was conducted in \cite{berrone2022variational}, highlighting the critical role of test functions and numerical integration on the error decay rate. Several adaptive strategies for VPINNs have been developed, incorporating error estimators to guide test-space refinements. 
To address the dependence on the choice of basis functions, Berrone et al. \cite{berrone2024meshfree} proposed a mesh-free approach in which the test-space mesh is no longer required. Instead, an a posteriori error estimator is employed to enrich the test space by incorporating new test functions. Additionally, Hu et al. \cite{HU2026118876} introduced an error estimator for stationary Navier-Stokes equations to guide adaptive mesh refinement within the VPINN framework. A rigorous a posteriori analysis was carried out in \cite{fuhrer2025posteriori}, with applications to adaptive refinement for neural network approximations. 
Beyond test space refinement, adaptive sampling of collocation points based on PDE residuals \cite{gao2024failure,gao2023failure,tang2023pinns,zhang2025annealed} has been studied to enhance the numerical accuracy of PINNs. These adaptive strategies aim to improve accuracy and convergence rates for neural network-based PDE solvers, particularly when approximating solutions with singularities and sharp gradients.

In RVPINNs, robustness is ensured only if the remainder is kept under control. 
In practice, this means that for a fixed mesh, continued training drives the neural network toward a solution determined by the discrete test space, rather than the true solution of the continuous variational formulation. 
If the mesh is too coarse, it might not be enough to control the true error during the training. A fully robust control can be obtained by incorporating a computable upper bound for the remainder into the loss \cite{fuhrer2025posteriori}. However, this implies a major training cost as a natural consequence of including additional terms to the loss function. Another way to ensure full control is to use a sufficiently fine mesh; however, this becomes impractical in complex scenarios involving singularities or steep gradients. 

We aim to accelerate convergence for RVPINNs in these complex problems. We extend the RVPINN framework with an adaptive test-space refinement strategy. 
This approach ensures robust training throughout this process without incorporating extra control terms into the loss function. The loss function is formulated in terms of the discrete Riesz representative of the variational residual, maintaining consistency with the original RVPINN methodology. 
The boundary term, when necessary, is defined by the trace norm of the boundary residual, following the approach in \cite{fuhrer2025posteriori}. To ensure that the discrete Riesz representative accurately approximates its continuous counterpart for a given neural network approximation, we monitor the discrepancy between the discrete and continuous Riesz representatives during training and refine the mesh whenever this discrepancy exceeds a prescribed tolerance.
The computable refinement indicator is based on the local difference between the discrete Riesz representatives computed in the original test space and in an enriched test space with a higher polynomial degree.

The main contributions of this work are as follows. We present upper bounds on the approximation error for the RVPINN loss function and prove their equivalence. 
We establish theoretical adaptive algorithms for RVPINNs with their convergence results.
Moreover, we propose a computable refinement indicator derived from two discrete Riesz representatives associated with two nested test spaces. We show that, under the saturation assumption \cite[Section 5.2]{ainsworth2000posteriori}, this indicator is a reliable and efficient error estimator for the discrepancy between the discrete and continuous Riesz representatives. We develop a practical test-space adaptive RVPINN algorithm and demonstrate its effectiveness by solving a Poisson problem on a square domain with a smooth solution, an elliptic interface problem with a discontinuous coefficient, and a Poisson equation on an L-shaped domain with a singular solution. 

Our adaptive strategies are motivated by the double adaptivity paradigm of Demkowicz et al. \cite{demkowicz2021double}. To make a connection with their framework, we note that their inner loop corresponds to the test space adaptation in our adaptive RVPINN strategy. However, while the outer loop in the double adaptivity paradigm refines a discrete trial mesh, it is replaced in our framework by neural network training, which optimizes the approximate solution within the fixed trial space parameterized by the network.

The remainder of this article is structured as follows. In Section \ref{sec: model}, we provide the notation and motivational problem for this study. In Section \ref{Sec:Pre}, we introduce the variational problem of PDEs and the neural network framework used throughout this work. We review the standard RVPINN methodology and discuss the loss function for problems in which boundary conditions are imposed weakly via an additional residual. In Section \ref{sec: Error estimates}, we
provide two types of error upper bounds and 
prove their equivalence. In Section \ref{sec: ideal adap}, we introduce idealized and sequential adaptivity algorithms and derive their error estimates. These results guarantee convergence of neural network training with RVPINNs under certain conditions, using the proposed adaptive strategies. In addition, we introduce the refinement indicator proposed for the adaptive framework.  
In Section \ref{sec:results}, we present a practical algorithm for the test-space adaptive RVPINNs and discuss the numerical experiments validating the proposed methodology. Section \ref{Sec:Conc}  summarizes the conclusions obtained from this work and outlines possible directions for future research.

\section{Motivational problem}\label{sec: model}

In this section, we introduce the notation used throughout the paper and present a general second-order elliptic Dirichlet problem that motivates the variational settings considered later. We first recall the classical weak formulation with strongly imposed boundary conditions. We then describe an alternative formulation in which the Dirichlet condition is imposed weakly through an additional boundary residual. The specific benchmark problems used in the numerical experiments are described later in Section~\ref{sec:results}.

\subsection{Notation}

Let $\Omega\subset\mathbb{R}^d$, $d\ge 1$, be a bounded Lipschitz domain with boundary $\Gamma:=\partial\Omega$. We denote by $L^2(\Omega)$ the space of real-valued square-integrable functions on $\Omega$, equipped with the inner product $(\cdot,\cdot)_{L^2(\Omega)}$ and norm $\|\cdot\|_{L^2(\Omega)}$. We also use $L^2(\Omega;\mathbb{R}^d)$ for vector-valued square-integrable functions, with the inner product and norm
\begin{equation*}
    (\mathbf w,\mathbf v)_{L^2(\Omega;\mathbb{R}^d)}
    :=\int_\Omega \mathbf w\cdot \mathbf v\,dx,
    \qquad \|\mathbf{v}\|_{L^2(\Omega;\mathbb{R}^d)}:=(\mathbf{v},\mathbf{v})^{1/2}_{L^2(\Omega;\mathbb{R}^d)}.
\end{equation*}
For $s\ge 0$, $H^s(\Omega)$ denotes the usual Sobolev space. In particular,
\begin{equation*}
    H^1(\Omega):=\{v\in L^2(\Omega):\nabla v\in L^2(\Omega;\mathbb{R}^d)\},
\end{equation*}
with the norm
\begin{equation*}
    \|v\|_{H^1(\Omega)}^2
    :=\|v\|_{L^2(\Omega)}^2+\|\nabla v\|_{L^2(\Omega;\mathbb{R}^d)}^2.
\end{equation*}
The subspace of functions with vanishing trace is denoted by
\begin{equation*}
    H_0^1(\Omega):=\{v\in H^1(\Omega):v|_\Gamma=0\},
\end{equation*}
and is equipped with the energy the inner product and norm
\begin{equation*}
    (v,w)_{H_0^1(\Omega)}:=(\nabla v,\nabla w)_{L^2(\Omega;\mathbb{R}^d)},
    \qquad
    \|v\|_{H_0^1(\Omega)}:=\|\nabla v\|_{L^2(\Omega;\mathbb{R}^d)}.
\end{equation*}
We also use the space
\begin{equation*}
    \mathbf H(\mathrm{div};\Omega)
    :=\{\mathbf v\in L^2(\Omega;\mathbb{R}^d):\nabla\cdot\mathbf v\in L^2(\Omega)\},
\end{equation*}
with the inner product
\begin{equation*}
    (\mathbf w,\mathbf v)_{\mathbf H(\mathrm{div};\Omega)}
    :=(\mathbf w,\mathbf v)_{L^2(\Omega;\mathbb{R}^d)}
      +(\nabla\cdot\mathbf w,\nabla\cdot\mathbf v)_{L^2(\Omega)}.
\end{equation*}
For $u\in H^1(\Omega)$, the trace $u|_\Gamma$ can be characterised through the generalized Green identity
\begin{equation*}
    \langle u|_\Gamma,\mathbf v\rangle_\Gamma
    :=(\nabla\cdot\mathbf v,u)_{L^2(\Omega)}
      +(\mathbf v,\nabla u)_{L^2(\Omega;\mathbb{R}^d)},
    \qquad \mathbf v\in\mathbf H(\mathrm{div};\Omega).
\end{equation*}
Here $\langle\cdot,\cdot\rangle_\Gamma$ denotes the dual pairing between the trace space $H^{1/2}(\Gamma):=\{u|_\Gamma:u \in H^1(\Omega)\}$ and its dual, realised through $\mathbf H(\mathrm{div};\Omega)$. We use the trace norm
\begin{equation*}
    \|\widehat u\|_{1/2,\Gamma}
    :=\sup_{0\neq\mathbf v\in\mathbf H(\mathrm{div};\Omega)}
    \frac{\langle \widehat u,\mathbf v\rangle_\Gamma}{\|\mathbf v\|_{\mathbf H(\mathrm{div};\Omega)}},
    \qquad \widehat u\in H^{1/2}(\Gamma).
\end{equation*}

For a triangulation $\mathcal T_k$ of $\Omega$, $\mathbb P^p(T)$ denotes the space of polynomials of total degree at most $p$ on an element $T\in\mathcal T_k$, and
\begin{equation}\label{Eq: P_o^p}
    \mathbb P_0^p(\mathcal T_k)
    :=\{v\in C^0(\overline\Omega):v|_T\in\mathbb P^p(T)\ \forall T\in\mathcal T_k,\ v|_\Gamma=0\}.
\end{equation}
We denote by $\mathbf{RT}^p(\mathcal T_k)$ the Raviart--Thomas space of order $p$ on $\mathcal T_k$ (cf. \cite[Section 2.3.1]{boffi2013mixed}).

\subsection{A model elliptic Dirichlet problem}

Let $a\in L^\infty(\Omega)$ satisfy the uniform ellipticity condition
\begin{equation}\label{eq:ellipticity_a}
    0<a_{\min}\le a(x)\le a_{\max}<\infty
    \qquad \text{for a.e. }x\in\Omega.
\end{equation}
Given $f\in L^2(\Omega)$ and $g\in H^{1/2}(\Gamma)$, we consider the model problem: find $u$ such that
\begin{equation}\label{eq:model_elliptic_pde}
    \begin{cases}
        \begin{aligned}
        -\nabla\cdot (a\nabla  u) &= f &&\textrm{in } \Omega, \\
        u &= g &&\textrm{on } \Gamma.
        \end{aligned}
    \end{cases}
\end{equation}
This problem includes the Poisson examples used later by taking $a\equiv 1$, and it also covers the interface example by allowing $a$ to be piecewise constant.

\subsection{Classical formulation with strongly imposed boundary conditions}

Let $w_g\in H^1(\Omega)$ be a lifting of the boundary datum, i.e., $w_g|_\Gamma=g$. Writing $u=u_0+w_g$ with $u_0\in H_0^1(\Omega)$, the classical weak formulation is: find $u_0\in H_0^1(\Omega)$ such that
\begin{equation}\label{eq:strong_bc_model}
    b_0(u_0,v)=l_0(v)
    \qquad \forall v\in H_0^1(\Omega),
\end{equation}
where
\begin{equation}\label{eq:formulation_strongly_impose}
    b_0(w,v):=(a\nabla w,\nabla v)_{L^2(\Omega;\mathbb R^d)},
    \qquad
    l_0(v):=(f,v)_{L^2(\Omega)}-(a\nabla w_g,\nabla v)_{L^2(\Omega;\mathbb R^d)}.
\end{equation}
For this formulation, it is natural to consider the weighted energy inner product and norm
\begin{equation}\label{eq:weighted_energy_inner_product}
    (w,v)_{a,\Omega}:=(a^{1/2}\nabla w,a^{1/2}\nabla v)_{L^2(\Omega;\mathbb R^d)},
    \qquad
    \|v\|_{a,\Omega}:=(v,v)_{a,\Omega}^{1/2}.
\end{equation}
The norm $\|\cdot\|_{a,\Omega}$ is equivalent to the standard $H_0^1(\Omega)$ energy norm by the ellipticity bounds in \eqref{eq:ellipticity_a}.
We take $U=V=H_0^1(\Omega)$ and $\|v\|_U=\|v\|_V=\|v\|_{a,\Omega}$.
With this choice, the bilinear form is exactly the weighted inner product $b_0(w,v)=(w,v)_{a,\Omega}$.
Consequently,
\begin{equation}\label{eq:bound_inf_sup_strong_impose}
    b_0(w,v)\le \|w\|_{a,\Omega}\|v\|_{a,\Omega},
    \qquad \textrm{and} \qquad
    \sup_{0\neq v\in H_0^1(\Omega)}\frac{b_0(w,v)}{\|v\|_{a,\Omega}}=\|w\|_{a,\Omega}.
\end{equation}

\subsection{Weak imposition through a composite residual}\label{sec: composite residual}

Strongly enforcing the Dirichlet condition through the neural network architecture can be inconvenient in complex domains or for non-homogeneous boundary data. An alternative is to keep the trial space as $U=H^1(\Omega)$ and impose the boundary condition weakly through an additional residual. We equip $U$ with the boundary-augmented weighted energy norm
\begin{equation}\label{eq:U_norm_weak_bc}
    \|w\|_U^2
    :=\|w\|_{a,\Omega}^2
      +\|w|_\Gamma\|_{1/2,\Gamma}^2,
    \qquad w\in H^1(\Omega),
\end{equation}
which is equivalent to the usual $H^1(\Omega)$ norm on bounded Lipschitz domains. Define
\begin{equation*}
    V_\Omega:=H_0^1(\Omega),
    \qquad
    V_\Gamma:=\mathbf H(\mathrm{div};\Omega),
    \qquad
    V:=V_\Omega\times V_\Gamma,
\end{equation*}
where the interior test space is equipped with the same weighted energy norm,
\begin{equation*}
    \|v\|_{V_\Omega}:=\|v\|_{a,\Omega}
    =\|a^{1/2}\nabla v\|_{L^2(\Omega;\mathbb R^d)}.
\end{equation*}
The product norm is
\begin{equation*}
    \|(v,\mathbf q)\|_V^2
    :=\|v\|_{a,\Omega}^2+\|\mathbf q\|_{\mathbf H(\mathrm{div};\Omega)}^2.
\end{equation*}
Strictly speaking, in the boundary component, only the normal trace of $\mathbf {q} $ on $\Gamma$ is seen by the residual. Thus, $V_\Gamma$ may equivalently be understood as the quotient of $\mathbf H(\mathrm{div};\Omega)$ by the kernel of the normal-trace map. In the numerical method, Raviart--Thomas functions are used as concrete representatives of this quotient space.

For $w\in H^1(\Omega)$ and $(v,\mathbf q)\in V$, set
\begin{align}
    b(w,(v,\mathbf q))
    &:=(a\nabla w,\nabla v)_{L^2(\Omega;\mathbb R^d)}
      +\langle w|_\Gamma,\mathbf q\rangle_\Gamma,\label{eq:composite_b_l_model}\\
      l(v,\mathbf q)
    &:=(f,v)_{L^2(\Omega)}+\langle g,\mathbf q\rangle_\Gamma.\nonumber
\end{align}
The corresponding residual is
\begin{equation}\label{eq:composite_residual_model}
    r(w,(v,\mathbf q))
    =r_\Omega(w,v)+r_\Gamma(w,\mathbf q),
\end{equation}
where
\begin{equation*}
    r_\Omega(w,v):=(f,v)_{L^2(\Omega)}-(a\nabla w,\nabla v)_{L^2(\Omega;\mathbb R^d)},
    \qquad
    r_\Gamma(w,\mathbf q):=\langle (g-w)|_\Gamma,\mathbf q\rangle_\Gamma.
\end{equation*}
The weakly imposed formulation is: find $u\in H^1(\Omega)$ such that
\begin{equation}\label{eq:weak_bc_model}
    r_\Omega(u,v)=0\quad\forall v\in H_0^1(\Omega),
    \qquad
    r_\Gamma(u,\mathbf q)=0\quad\forall \mathbf q\in\mathbf H(\mathrm{div};\Omega).
\end{equation}
The second equation enforces $u|_\Gamma=g$ in $H^{1/2}(\Gamma)$. Therefore, \eqref{eq:weak_bc_model} is equivalent to the classical Dirichlet problem \eqref{eq:model_elliptic_pde}.

By the Cauchy--Schwarz inequality in the weighted energy inner product and by the definition of the trial space norm and the product test norm, we obtain
\begin{align}
    b(w,(v,\mathbf q))
    &\le \|w\|_{a,\Omega}\|v\|_{a,\Omega}
      +\|w|_\Gamma\|_{1/2,\Gamma}\|\mathbf q\|_{\mathbf H(\mathrm{div};\Omega)}
    \le \|w\|_U\,\|(v,\mathbf q)\|_V.
    \label{eq:composite_boundedness}
\end{align}
In addition, let $\widehat w:=w|_\Gamma$ and let $H_a\widehat w\in H^1(\Omega)$ denote the $a$-harmonic lifting of $\widehat w$, defined by
\begin{equation*}
    (H_a\widehat w)|_\Gamma=\widehat w
    \qquad \textrm{and} \qquad
    (a\nabla H_a\widehat w,\nabla v)_{L^2(\Omega;\mathbb R^d)}=0
    \quad\forall v\in H_0^1(\Omega).
\end{equation*}
Set $w_0:=w-H_a\widehat w\in H_0^1(\Omega)$. Then
\begin{equation*}
    b(w,(v,\mathbf q))
    =(a^{1/2}\nabla w_0,a^{1/2}\nabla v)_{L^2(\Omega;\mathbb R^d)}
    +\langle \widehat w,\mathbf q\rangle_\Gamma.
\end{equation*}
By the Cauchy--Schwarz inequality, we obtain
\begin{align*}
    b(w,(v,\mathbf q)) &\leq \|w_0\|_{a,\Omega}\|v\|_{a,\Omega} + \|\widehat w\|_{1/2,\Gamma}\|\mathbf{q}\|_{\mathbf{H}(\textrm{div};\Omega)} \\
    &\leq \left(
    \|w_0\|_{a,\Omega}^2
    +\|\widehat w\|_{1/2,\Gamma}^2
    \right)^{1/2}\|(v,\mathbf q)\|_V.
\end{align*}
Thus, the supremum over $V=V_\Omega\times V_\Gamma$ gives
\begin{align}
    \sup_{0\neq (v,\mathbf q)\in V}
    \frac{b(w,(v,\mathbf q))}{\|(v,\mathbf q)\|_V}
    &=
    \left(
    \|w_0\|_{a,\Omega}^2
    +\|\widehat w\|_{1/2,\Gamma}^2
    \right)^{1/2}.
    \label{eq:composite_infsup_intermediate}
\end{align}
The $a$-harmonic lifting is stable in the weighted energy norm. In particular, there exists $C_H>0$, depending only on $\Omega$ and the ellipticity bounds of $a$, such that
\begin{equation*}
    \|H_a\widehat w\|_{a,\Omega}\le C_H\|\widehat w\|_{1/2,\Gamma}.
\end{equation*}
Moreover, $w_0$ is $a$-orthogonal to $H_a\widehat w$, and therefore
\begin{equation*}
    \|w\|_{a,\Omega}^2
    =\|w_0\|_{a,\Omega}^2+\|H_a\widehat w\|_{a,\Omega}^2.
\end{equation*}
Consequently,
\begin{equation*}
    \|w\|_U
    \le C
    \left(
    \|w_0\|_{a,\Omega}^2
    +\|\widehat w\|_{1/2,\Gamma}^2
    \right)^{1/2},
    \qquad
    C:=\sqrt{\max\{1,1+C_H^2\}}.
\end{equation*}
Combining this estimate with \eqref{eq:composite_infsup_intermediate} gives
\begin{equation}\label{eq:composite_infsup}
    \sup_{0\neq (v,\mathbf q)\in V}
    \frac{b(w,(v,\mathbf q))}{\|(v,\mathbf q)\|_V}
    \ge C^{-1}\|w\|_U.
\end{equation}

\section{Robust variational physics-informed neural networks (RVPINNs)}\label{Sec:Pre} 

\subsection{Abstract framework}\label{sec:abstract} 

Let $U$ and $V$ be Hilbert spaces equipped with norms $\|\cdot\|_U$ and $\|\cdot\|_V$, respectively. Let $V'$ denote the dual space of $V$. We consider a PDE problem that admits the following variational formulation:
\begin{equation}\label{eq:weakPDE}
    \text{Find } u \in U \text{ such that } b(u,v) = l(v), \quad \forall v \in V,
\end{equation}
where $b(\cdot,\cdot)$ is a bilinear form in $U \times V$ and $l(\cdot)\in V'$ is a bounded linear functional satisfying the following assumptions:
\begin{enumerate}
    \item Boundedness of $b(\cdot,\cdot)$: there is a constant $\mu>0$ such that
    \begin{equation}
        b(w,v) \leq \mu\|w\|_U\|v\|_V \quad \forall w \in U, v \in V, \label{eq:bound_cont}
    \end{equation}
    \item Inf-sup stability: there is a constant $\alpha>0$ such that
    \begin{equation}
        \sup\limits_{0 \neq v \in V} \dfrac{b(w,v)}{\|v\|_V} \geq \alpha \|w\|_U \quad \forall w \in U, \label{eq:inf_sup_cont}
    \end{equation} 
    \item Adjoint injectivity: \begin{equation}
        b(w,v) = 0 \quad \forall w \in U \Longrightarrow v = 0. \label{eq:A_prime_kernel}
    \end{equation}
\end{enumerate}

Under these assumptions, Problem~\eqref{eq:weakPDE} admits a unique solution $u\in U$ by the Banach–Ne\v{c}as–Bab\v{u}ska Theorem. 
Furthermore, the conditions~\eqref{eq:inf_sup_cont} and \eqref{eq:A_prime_kernel} are well-known to imply the adjoint inf-sup condition (see, e.g., \cite[Theorem 1]{demkowicz2006babuvska}):
\begin{equation}\label{eq:inf_sup_cont_dual}
\sup_{0 \neq w \in U} \dfrac{b(w, v)}{\|w\|_U} \geq \alpha \|v\|_V, \quad \forall \, v \in V.
\end{equation}

\begin{remark}
    For the motivational problem, formulation \eqref{eq:formulation_strongly_impose} satisfies these assumptions with constants $\mu=1$ and $\alpha=1$ when the weighted energy norm is used (see \eqref{eq:bound_inf_sup_strong_impose}). The boundedness and inf-sup stability assumptions for the composite bilinear form \eqref{eq:composite_b_l_model} are satisfied with $\mu=1$ and $\alpha = C^{-1}$ (see \eqref{eq:composite_boundedness} and \eqref{eq:composite_infsup}). The adjoint injectivity assumption also holds: if $b(w,(v,\mathbf q))=0$ for all $w\in H^1(\Omega)$, then testing first with $w\in H_0^1(\Omega)$ gives $v=0$ by coercivity in the weighted inner product, and the remaining boundary relation gives $\mathbf q=0$ in the quotient trace space.
\end{remark}

Let $(\cdot,\cdot)_V$ be the inner product inducing the norm $\|\cdot\|_V$ and define the dual norm of $V'$ by
\begin{equation}\label{eq:dual_norm_continuous}
\|l(\cdot)\|_{V^\prime} := \sup_{0 \neq v \in V} \dfrac{l(v)}{\|v\|_V}.
\end{equation} 
For any $w \in U$, we define $r(w,\cdot) := l(\cdot)-b(w,\cdot) \in V'$ as the variational residual with respect to $w$. According to the Riesz-Fr\'echet representation theorem (see~\cite[Theorem 5.5]{brezis2011functional}), there exists a unique Riesz representative $\phi(w)\in V$ such that:
\begin{equation}\label{eq:Riesz}
\left(\phi(w), v\right)_V = r(w, v), \quad \forall \, v \in V, \quad \text{ and } \quad \|\phi(w)\|_V = \|r(w, \cdot)\|_{V^\prime}.
\end{equation}
As a consequence of the boundedness of $b(\cdot,\cdot)$ and the inf-sup stability, the following inequality can be proven:
\begin{equation}\label{eq:robust estimation}
    \dfrac{1}{\mu} \|\phi(w)\|_{V} \leq \|u-w\|_U \leq \dfrac{1}{\alpha} \|\phi(w)\|_{V}.
\end{equation}
Inequality \eqref{eq:robust estimation} implies that $\|\phi(w)\|_V$ is a robust estimator of the true error $\|u-w\|_U$.

\subsection{Neural network framework}

To approximate the solution to Problem \eqref{eq:weakPDE}, we employ a Deep Neural Network (DNN) $u_\theta$ parametrized by $\theta \in \mathbb{R}^S$, which comprises the network's weights and biases. The neural network maps an input vector $\Vx = (x_1,\dots,x_d)$ to a scalar output $u_{\theta}(\Vx)$. In this study, we use a fully connected feedforward architecture consisting of $L$ layers. 
In this architecture, each layer $j \in \{ 1, \dots, L-1\}$ consists of neurons that compute a weighted sum of their inputs plus a bias, followed by a nonlinear activation function $\sigma$ (e.g., the hyperbolic tangent). Specifically, the output of the $j$-th layer is given by:
\begin{equation}
\Vz^{(j)} = \sigma(\Vw^{(j)} \Vz^{(j-1)} + \Vb^{(j)}), \quad j=1,\dots,L-1,
\end{equation}
where $\Vw^{(j)}$ and $\Vb^{(j)}$ are the weight matrix and bias vector for layer $j$, respectively, and $\Vz^{(0)} = \Vx$ is the network input. The final layer $L$ utilizes the identity function as its activation to produce the network output:
\begin{equation}
\bar u_{\theta}(\Vx) =  \Vw^{(L)} \Vz^{(L-1)} + \Vb^{(L)}.
\end{equation}
When homogeneous Dirichlet boundary conditions are imposed strongly on the neural network architecture, 
\begin{equation*}
u_\theta(\Vx) := \beta(\Vx) \, \bar u_{\theta}(\Vx),
\end{equation*}
where \(\beta\) is a prescribed function that vanishes on \(\Gamma\), otherwise \(\beta \equiv 1\). The optimal parameters $\theta$ are identified by minimizing a suitable loss functional through iterative training.

We denote the set of all functions representable by this DNN architecture as $U_{NN}$, which we assume to be a subset of the trial space $U$:
\begin{equation}\label{eq:UNN}
U_{NN} := \{u_{\theta} \, : \, \theta \in \mathbb{R}^s \} \subset U.
\end{equation}

\subsection{Loss function}\label{Sec:RVPINNs}

Let $V_k = \mathrm{span}\{\varphi_n\}_{n=1}^{N_k}$ be a finite-dimensional subspace of $V$.
For a given neural network $u_\theta$, the discrete Riesz representative of the residual $r(u_{\theta},\cdot)$ with respect to $V_k$ is the function $\phi_k(u_{\theta}) \in V_k$ satisfying 
\begin{equation}\label{eq:discreteRiesz}
(\phi_k(u_{\theta}),  v_k)_{V} = r(u_{\theta}, v_k), \quad \forall \, v_k\in V_k.
\end{equation}
It follows immediately that the discrete dual norm of the residual is equal to the $V$-norm of its discrete Riesz representative. Indeed,
\begin{align}\label{Eq: discrete dual norm of residual}
    \|r(u_\theta,\cdot)\|_{V_k'} = \sup_{0\neq v_k \in V_k}\dfrac{r(u_{\theta},v_k)}{\|v_k\|_V} = \sup_{0\neq v_k \in V_k}\dfrac{(\phi_k(u_{\theta}),v_k)_V}{\|v_k\|_V} = \|\phi_k(u_{\theta})\|_{V}.
\end{align}
Moreover, a straightforward computation shows that the norm of the discrete Riesz representative provides a lower bound for the approximation error. That is, one has
\begin{equation}
\dfrac{1}{\mu} \|\phi_k(u_{\theta})\|_V \leq \|u-u_{\theta}\|_U.
\end{equation}

The RVPINN loss function is formulated using the norm of the discrete Riesz representative \eqref{Eq: discrete dual norm of residual} of the residual corresponding to $u_\theta$. 
Following \cite{rojas2024robust}, the loss is defined as:
\begin{equation}\label{eq:loss_RVPINNs}
\mathcal{L}_k\left(u_{\theta} \right) := \|\phi_k(u_{\theta})\|^2_{V} + C(u_{\theta}),
\end{equation}
where $C(\cdot)$ imposes constraints on the neural network (e.g., data interpolation). 
As a consequence of the linearity of the residual $r(u_\theta,\cdot)$, and of the inner product $(\cdot,\cdot)_V$, the discrete Riesz representative $\phi_k(u_\theta)$ can be written as
\begin{equation}\label{eq:linear_comb}
\phi_k(u_{\theta}) := \sum_{n=1}^{N_k} \eta_\theta(n) \varphi_n, 
\end{equation}
with coefficients $\eta_\theta(n) \in \mathbb{R}$. Substituting \eqref{eq:linear_comb} into \eqref{eq:discreteRiesz} yields the linear system
\begin{equation}\label{eq:Riesz_mat}
G_k \bm{\eta}_\theta = \bm{r}_{\theta,k},
\end{equation}
where $\bm{\eta}_\theta$ is the vector of coefficients $\eta_\theta(n)$, $G_k$ is the ($\theta$-independent) symmetric and positive definite Gram matrix $G_k(n,m) = (\varphi_m,\varphi_n)_{V}$, and $\bm{r}_{\theta,k}$ is the vector with entries $\bm{r}_{\theta,k}(n) = r(u_{\theta}, \varphi_n)$, $n=1,\dots,N_k$.
From \eqref{eq:linear_comb} and the definitions of $G_k$ and $\bm{r}_{\theta,k}$, the norm of the discrete Riesz representative satisfies
\begin{equation}
\|\phi_k(u_{\theta})\|^2_{V} = (\phi_k(u_{\theta}) , \phi_k(u_{\theta}))_V = \bm{r}_{\theta,k}^T G_k^{-1} \bm{r}_{\theta,k}.
\end{equation}
Hence, the loss functional \eqref{eq:loss_RVPINNs} can be equivalently written as
\begin{equation}\label{eq:AVPINNs_loss_mat}
\mathcal{L}_k(u_{\theta}) = \bm{r}_{\theta,k}^T G_k^{-1} \bm{r}_{\theta,k} + C(u_{\theta}).
\end{equation}
For simplicity, throughout this article, we assume that the loss function $\mathcal{L}_k(u_\theta)$ has $C(u_\theta) = 0$. 

\begin{remark}
    The inversion of $G_k$ can be computationally expensive when using a standard piecewise polynomial test space, as typically found in FEM. In this scenario, maintaining efficiency may require the use of preconditioners, as discussed in \cite{fuhrer2025posteriori}. However, a remarkable property of the RVPINN framework is that the inversion of $G_k$ becomes cost-free or trivial when choosing a localizable test space, such as an orthonormal basis or discontinuous test functions.
\end{remark}

\begin{remark}
    To apply the RVPINN framework to the composite residual formulation introduced in Section \ref{sec: composite residual}, we let $V_k^\Omega := \mathrm{span}\{\varphi_n^\Omega\}_{n=1}^{N_k}$ and $V_k^\Gamma := \mathrm{span}\{\varphi_m^\Gamma\}_{m=1}^{M_k}$ be discrete subspaces of $V_\Omega$ and $V_\Gamma$, respectively. By employing the product norm, the corresponding loss function is defined as:
\begin{equation}\label{Eq: loss combined}
\mathcal{L}_k\left(u_{\theta} \right) := \|\phi_{k}^\Omega(u_{\theta})\|^2_{V_\Omega} + \|\phi_{k}^\Gamma(u_{\theta})\|^2_{V_\Gamma},
\end{equation}
where $\phi_{k}^\Omega(u_{\theta})$ and $\phi_{k}^\Gamma(u_{\theta})$ are discrete Riesz representatives of the residuals $r_\Omega(u_{\theta},\cdot)$ and $r_\Gamma(u_{\theta},\cdot)$, respectively (see \eqref{eq:discreteRiesz}).
This formulation is computationally equivalent to:
\begin{equation}\label{Eq: loss two terms}
\mathcal{L}_k(u_{\theta}) = (\bm{r}_{\theta,k}^\Omega)^T (G_{k}^\Omega)^{-1} \bm{r}_{\theta,k}^\Omega + (\bm{r}_{\theta,k}^\Gamma)^T (G_{k}^\Gamma)^{-1} \bm{r}_{\theta,k}^\Gamma,
\end{equation}
where $\bm{r}_{\theta,k}^\Omega$ and $\bm{r}_{\theta,k}^\Gamma$ are vectors with entries $\bm{r}_{\theta,k}^\Omega(n) = r_\Omega(u_{\theta}, \varphi_n^\Omega)$, and $\bm{r}_{\theta,k}^\Gamma(m) = r_\Gamma(u_{\theta}, \varphi_m^\Gamma)$. Here, $G_k^{\Omega}$ and $G_k^{\Gamma}$ denote the corresponding Gram matrices for the domain and boundary test spaces, respectively.
\end{remark}

\section{Error upper bounds}\label{sec: Error estimates}

For a given neural network approximation $u_\theta$, we let $\phi(u_{\theta})\in V $ denote the (continuous) Riesz representative (see \eqref{eq:Riesz}) associated with $u_{\theta}$ satisfying
    \begin{equation}\label{optimal discrete Riesz}
        (\phi(u_{\theta}), v)_V = r(u_{\theta}, v) \quad \forall v \in V.
    \end{equation}
In addition, we recall the discrete Riesz representative $\phi_k(u_\theta)$ with respect to $V_k \subset V$ (see \eqref{eq:discreteRiesz}). In this section, we discuss two types of error upper bounds: the first is based on a residual-remainder term, while the second relies on the discretization error of the Riesz representative.

\subsection{Two error upper bounds}

As shown in \cite{fuhrer2025posteriori}, suppose there exists an operator $\Pi_k:V\to V_k$ and a constant $C_{\Pi_k}>0$ such that $\|\Pi_kv\|_V \leq C_{\Pi_k}\|v\|_V$ for all $v \in V$. Then, $$ \|\phi(u_\theta)\|_V \leq C_{\Pi_k}\|\phi_k(u_{\theta})\|_V + \mathcal{R}_k(u-u_{\theta}),$$
where 
$\mathcal{R}_k(u-u_{\theta}):=\sup\limits_{0\neq v \in V}\dfrac{b(u-u_{\theta},v-\Pi_k v)}{\|v\|_V}$ is a residual remainder associated with the discrete space $V_k$. Applying \eqref{eq:robust estimation}, we obtain the following upper bound:
\begin{equation}\label{Eq: error bound osc}
    \|u-u_\theta\|_U \leq \frac{1}{\alpha}\left(C_{\Pi_k}\|\phi_k(u_{\theta})\|_V + \mathcal{R}_k(u-u_{\theta})\right).
\end{equation}

Alternatively, by using Inequality \eqref{eq:robust estimation} and the triangle inequality, we derive an upper bound for the true error:
\begin{equation}\label{Eq: true error upper bound}
    \|u-u_\theta\|_U \leq \frac{1}{\alpha}\|\phi(u_\theta)\|_V \leq \dfrac{1}{\alpha}\left(\|\phi_k(u_\theta)\|_V + \|\phi(u_\theta)-\phi_k(u_\theta)\|_V\right).
\end{equation}

Inequality \eqref{Eq: true error upper bound} highlights that bounding the error between the trained neural network $u_\theta$ and the exact solution $u$ depends on two ingredients: 
    \begin{enumerate}
        \item Optimization convergence: The training process must successfully minimize the discrete dual norm, ensuring that $\|\phi_k(u_{\theta})\|_V$ is small. This reflects the quality of the nonlinear solver (e.g., Adam \cite{kingma2014adam}) in finding a near-optimal discrete minimizer.
        \item Accuracy of the discrete test space: The discrete test space $V_k$ must be rich enough to ensure that the discrete Riesz representative $\phi_k(u_{\theta})$ is a faithful approximation of the continuous one, $\phi(u_{\theta})$. 
    \end{enumerate}

\subsection{Equivalence of two error upper bounds}

In this section, we establish the equivalence between the error upper bound based on the remainder term, $\|\phi_k(u_{\theta})\|_V + \mathcal{R}_k(u-u_{\theta})$, and the upper bound based on the Riesz representative error, $\|\phi_k(u_\theta)\|_V + \|\phi(u_\theta)-\phi_k(u_\theta)\|_V$.

\begin{lemma}\label{lem: Riesz disc identity}
    Let $u$ be the exact solution to \eqref{eq:weakPDE} and $u_\theta$ be a neural network approximation. Let $\phi(u_\theta)$ and $\phi_k(u_\theta)$ be the continuous and discrete Riesz representatives of the residual, as defined in \eqref{optimal discrete Riesz} and \eqref{eq:discreteRiesz}, respectively. Given an operator $\Pi_k:V\to V_k$ such that $\|\Pi_kv\|_V \leq C_{\Pi_k}\|v\|_V$ for all $v \in V$, the following identity holds:
    \begin{equation}
        (\phi(u_\theta)-\phi_k(u_\theta),v)_V = b(u-u_\theta,v-\Pi_kv)-(\phi_k(u_\theta),v-\Pi_kv)_V.
    \end{equation}
    \begin{proof}
        By decomposing the inner product, and using the definition of Riesz representatives $\phi(u_\theta)$ and $\phi_k(u_\theta)$, we obtain
        \begin{align*}
            (\phi(u_\theta)-\phi_k(u_\theta),v)_V &= (\phi(u_\theta),v)_V - (\phi_k(u_\theta),v-\Pi_kv)_V - (\phi_k(u_\theta),\Pi_kv)_V \\
            &= r(u_\theta, v) - (\phi_k(u_\theta),v-\Pi_kv)_V - r(u_\theta,\Pi_kv).
        \end{align*}
        Applying the definition of the residual $r(u_\theta,\cdot)$ and the bilinearity of $b(\cdot,\cdot)$, we have
        \begin{align*}
            (\phi(u_\theta)-\phi_k(u_\theta),v)_V 
            &= b(u-u_\theta, v) - (\phi_k(u_\theta),v-\Pi_kv)_V - b(u-u_\theta,\Pi_kv) \\
            &= b(u-u_\theta, v-\Pi_kv) - (\phi_k(u_\theta),v-\Pi_kv)_V,
        \end{align*}
        which completes the proof.
    \end{proof}
\end{lemma}

\begin{proposition}[Equivalence of two upper bounds]
    Let $u$ be the exact solution to \eqref{eq:weakPDE} and $u_\theta$ be a neural network approximation. Let $\phi(u_\theta)$ and $\phi_k(u_\theta)$ be the continuous and discrete Riesz representatives of the residual, as defined in \eqref{optimal discrete Riesz} and \eqref{eq:discreteRiesz}, respectively. For a given operator $\Pi_k:V\to V_k$ such that $\|\Pi_kv\|_V \leq C_{\Pi_k}\|v\|_V$ for all $v \in V$, let $D_{\Pi_k}:=\|I-\Pi_k\|$. Then, the following inequalities hold:
    \begin{equation}\label{Eq: equiv1}
        \|\phi(u_\theta)-\phi_k(u_\theta)\|_V + \|\phi_k(u_\theta)\|_V \leq \mathcal{R}_k(u-u_\theta) + (1+D_{\Pi_k})\|\phi_k(u_\theta)\|_V,
    \end{equation}
    and
    \begin{equation}\label{Eq: equiv2}
        \mathcal{R}_k(u-u_\theta) + \|\phi_k(u_\theta)\|_V \leq \|\phi(u_\theta)-\phi_k(u_\theta)\|_V+(1+D_{\Pi_k})\|\phi_k(u_\theta)\|_V.
    \end{equation}
\end{proposition}
\begin{proof}
    By the Cauchy-Schwarz inequality and the property of the operator norm for $I-\Pi_k$, we have
    \begin{align}\label{Eq: I-Pi_k}
        (\phi_k(u_\theta),v-\Pi_kv)_V \leq \|\phi_k(u_\theta)\|_V\|v-\Pi_kv\|_V 
        \leq D_{\Pi_k}\|\phi_k(u_\theta)\|_V\|v\|_V.
    \end{align}
    Applying Lemma \ref{lem: Riesz disc identity}, Equation \eqref{Eq: I-Pi_k}, and the definition of $\mathcal{R}_k(u-u_\theta)$, we obtain
    \begin{align*}
        \|\phi(u_\theta)-\phi_k(u_\theta)\|_V &= 
        \sup_{0\neq v \in V}\dfrac{(\phi(u_\theta)-\phi_k(u_\theta),v)_V}{\|v\|_V} \\
        &\leq \sup_{0\neq v \in V}\dfrac{b(u-u_\theta,v-\Pi_kv)}{\|v\|_V} + \sup_{0\neq v \in V}\dfrac{(\phi_k(u_\theta),v-\Pi_kv)_V}{\|v\|_V}\\
        &\leq \mathcal{R}_k(u-u_\theta) + D_{\Pi_k}\|\phi_k(u_\theta)\|_V.
    \end{align*}
    Adding $\|\phi_k(u_\theta)\|_V$ to both sides yields \eqref{Eq: equiv1}.
    
    Conversely, starting from the definition of $\mathcal{R}_k(u-u_\theta)$ and applying Lemma \ref{lem: Riesz disc identity} and Equation \eqref{Eq: I-Pi_k}, we have
    \begin{align*}
        \mathcal{R}_k(u-u_\theta) &= \sup_{0\neq v \in V}\dfrac{b(u-u_\theta,v-\Pi_kv)}{\|v\|_V} \\
        &\leq \sup_{0\neq v \in V} \dfrac{(\phi(u_\theta)-\phi_k(u_\theta),v)_V}{\|v\|_V} + \sup_{0\neq v \in V} \dfrac{(\phi_k(u_\theta),v-\Pi_k v)_V}{\|v\|_V} \\
        &\leq \|\phi(u_\theta)-\phi_k(u_\theta)\|_V+D_{\Pi_k}\|\phi_k(u_\theta)\|_V.
    \end{align*}
    Adding $\|\phi_k(u_\theta)\|_V$ to both sides yields \eqref{Eq: equiv2}.
\end{proof}

Although the $V$-norm of the (continuous) Riesz representative provides a robust estimate of the true error in the $U$-norm (see \eqref{eq:robust estimation}), replacing $\phi(u_{\theta})$ with the discrete Riesz representative $\phi_k(u_{\theta})$ does not, in general, yield an analogous result.
The error estimates in \eqref{Eq: error bound osc} and \eqref{Eq: true error upper bound} show that minimizing $\|\phi_k(u_{\theta})\|_V$ does not necessarily minimize the true error when the remainder term dominates. 
If the underlying discrete test space $V_k$ is too coarse, a reduction in the loss functional $\mathcal{L}_k(u_{\theta})$ does not necessarily guarantee a decrease in the true error. In practice, this issue can be observed after some training iterations: the loss decreases while the true approximation error stagnates (see \cite[Fig. 3]{rojas2024robust}).
Conversely, starting with an overly refined test-space mesh is computationally costly in terms of memory consumption, as it increases both the number of quadrature evaluations and the associated storage requirements per training iteration. While a fine initial mesh avoids the need to frequently redefine the loss function structure, it wastes resources on regions where the error is already negligible.

To address these issues, the neural network should be trained initially using a coarse test-space mesh, which is computationally inexpensive, and the mesh should then be refined adaptively whenever the loss becomes sufficiently small. 
The mesh refinement ensures that the remainder term decreases as the test space is enriched, allowing continued training of the neural network under the updated loss function.

\section{Theoretical adaptivity strategies}\label{sec: ideal adap}

In this section, we analyze theoretical adaptivity strategies designed to guarantee the convergence of neural network approximations within the RVPINN framework. We first discuss an idealized setting in which the approximation error can be bounded by a prescribed tolerance. We then consider a sequential adaptivity strategy based on the assumption that the trained neural network does not increase the Riesz discrepancy.

\subsection{Idealized adaptivity strategy}

Inequality \eqref{Eq: true error upper bound} shows that the true error can be upper bounded by two quantities: the discretization error of the test space and the minimized loss over that space. Therefore, convergence to the solution is ensured if the Riesz discrepancy 
\begin{align*}
    \eta_k(u_\theta):=\|\phi(u_\theta)-\phi_k(u_\theta)\|_V
\end{align*} 
is kept under control via test-space refinement, while the square root of the discrete loss $\sqrt{\mathcal{L}_k(u_\theta)}=\|\phi_k(u_\theta)\|_V$ is minimized through training. 

Given an initial mesh $\mathcal{T}_0$, $\epsilon_0 >0$, and $0<\delta<1$. Let $u_{\theta^0}$ be an initial neural network. We initialize $\mathcal{T}_1^0:=\mathcal{T}_{0}$ and $u_{\theta^{1}}^0:=u_{\theta^{0}}$. An idealized algorithm sets $k:=1$ and iterates:
\begin{align*}
    &\epsilon_k=\delta\epsilon_{k-1}\\
    &\textrm{For } i= ~1,2,\ldots,N\\
    & \quad \mathcal{T}_{k}^i = ~\mathrm{ADAPT}(\mathcal{T}_{k}^0,u_{\theta^{k}}^{i-1},\epsilon_{k}) \\
    & \quad u_{\theta^{k}}^{i} = ~\mathrm{LEARN}(\mathcal{T}_{k}^i,u_{\theta^{k}}^0,\epsilon_{k})\\
    &\textrm{End for}\\
    &\mathcal{T}_{k+1}^0:=\mathcal{T}_{k}^N, u_{\theta^{k+1}}^0:=u_{\theta^{k}}^N\\
    &k\leftarrow k+1.
\end{align*}

Here, the ADAPT process uses the trained neural network $u_{\theta^k}^{i-1}$ as prior knowledge to refine the initial mesh $\mathcal{T}_k^0$ and obtain $\mathcal{T}_k^i$ such that 
\begin{equation}\label{eq: ADAPT condition}
    \|\phi(u_{\theta^k}^{i-1})-\phi_{\mathcal{T}_{k}^i}(u_{\theta^k}^{i-1})\|_V \leq \epsilon_k,
\end{equation} 
where $\phi_{\mathcal{T}_k^i}(u_{\theta^k}^{i-1})$ is the discrete Riesz representative of $r(u_{\theta^k}^{i-1},\cdot)$ with respect to $\mathcal{T}_k^i$.
The LEARN process uses the refined mesh $\mathcal{T}_{k}^i$ as an input and trains the neural network $u_{\theta^k}^0$ over this mesh to obtain $u_{\theta^k}^i$ such that 
\begin{equation}\label{eq: LEARN condition}
    \|\phi_{\mathcal{T}_{k}^i}(u_{\theta^k}^{i})\|_V \leq \epsilon_k.
\end{equation}

Under the idealized adaptive algorithm that controls the Riesz discrepancy, and under sufficient training on each refined test space, the following
error bound holds.
\begin{proposition}[Error bound for idealized strategy]
    Let $u$ be the solution to \eqref{eq:weakPDE} and let $N=\infty$. Assuming that $u_{\theta^k}^i \to u_{\theta^k}$ as $i\to\infty$, we obtain:
    \begin{equation*}
        \|u-u_{\theta^k}\|_U
        \leq \dfrac{2}{\alpha}\delta^k\epsilon_0,
    \end{equation*}
    where $\alpha$ is the constant in Equation \eqref{eq:inf_sup_cont}.
\end{proposition}
\begin{proof}
    Applying Equation (\ref{Eq: true error upper bound}), we obtain:
    \begin{equation}\label{eq: idealized upper}
        \|u-u_{\theta^k}\|_U \leq \dfrac{1}{\alpha}\left(\|\phi_{\mathcal{T}_k^N}(u_{\theta^k})\|_V + \|\phi(u_{\theta^k})-\phi_{\mathcal{T}_k^N}(u_{\theta^k})\|_V\right).
    \end{equation}
    By the assumption, both $u_{\theta^k}^{i-1}$ and $u_{\theta^k}^i$ converge to $u_{\theta^k}$ as $i\to \infty$. Taking the limit in Equations
    (\ref{eq: ADAPT condition}) and (\ref{eq: LEARN condition}), and substituting the results into Equation (\ref{eq: idealized upper}), we obtain
    \begin{equation*}
        \|u-u_{\theta^k}\|_U \leq \dfrac{2}{\alpha}\epsilon_k.
    \end{equation*}
    Finally, using the relation $\epsilon_k=\delta\epsilon_{k-1}$ recursively yields $\epsilon_k=\delta^k\epsilon_0$, which completes the proof.
\end{proof}

Although this proposition guarantees convergence to the solution when $N=\infty$, the development of a practical algorithm to ensure such convergence requires further study. Alternatively, we propose a sequentially adaptive strategy with $N=1$, assuming that the Riesz discrepancy decreases after training, a behavior observed empirically in later sections. The following section details this adaptive algorithm and its corresponding theoretical results.

\subsection{Sequential adaptivity strategy}\label{sec: post-training adap}

Given an initial mesh $\mathcal{T}_0$, parameters $\epsilon_0>0$ and $0<\delta<1$, and an initial neural network $u_{\theta^0}$ satisfying $\sqrt{\mathcal{L}_{0}(u_{\theta^{0}})}\leq\epsilon_{0}$, a sequential adaptivity algorithm sets $k:=1$ and iterates:
\begin{align*}
    \epsilon_{k}&=\delta\epsilon_{k-1}\\
    \mathcal{T}_{k} &= \mathrm{ADAPT}(\mathcal{T}_{k-1},u_{\theta^{k-1}},\epsilon_{k}) \\
    u_{\theta^{k}} &= \mathrm{LEARN}(\mathcal{T}_{k},u_{\theta^{k-1}},\epsilon_{k}) \\
    k &\leftarrow k+1.
\end{align*}

Procedure descriptions

\begin{itemize}
    \item ADAPT: We refine the current mesh $\mathcal{T}_{k-1}$ to produce the refined mesh $\mathcal{T}_{k}$ such that the Riesz discrepancy satisfies $\eta_{k}(u_{\theta^{k-1}}) \leq \epsilon_{k}$.
    \item LEARN: We train the neural network $u_{\theta^{k-1}}$ on the mesh $\mathcal{T}_{k}$ by minimizing the loss function $\mathcal{L}_{k}$ to obtain the neural network $u_{\theta^{k}}$ such that $\sqrt{\mathcal{L}_{k}(u_{\theta^{k}})}\leq\epsilon_{k}$.
\end{itemize}

The following proposition establishes a theoretical optimality of the sequential algorithm, assuming that the neural network $u_{\theta^k}$ obtained after training does not increase the Riesz discrepancy compared to the pre-trained state $u_{\theta^{k-1}}$, i.e., $\eta_{k}(u_{\theta^{k}}) \leq \eta_{k}(u_{\theta^{k-1}})$. We note that this assumption is not an inherent property and may need to be monitored during training via a computable indicator.

\begin{proposition}[Error bound for sequential strategy]
    Let $u$ be the solution to \eqref{eq:weakPDE} and $u_{\theta^{k}} \in U_{NN}$ be the neural network obtained at the $k$th iteration of the sequential algorithm. If $\eta_{k}(u_{\theta^{k}}) \leq \eta_{k}(u_{\theta^{k-1}})$, then the following bound holds:
    \begin{equation*}
        \|u-u_{\theta^{k}}\|_U \leq \dfrac{2}{\alpha}\delta^k\epsilon_{0},
    \end{equation*}
    where $\alpha$ is the constant in Equation \eqref{eq:inf_sup_cont}.
\end{proposition}

\begin{proof}
    By the assumption of this proposition and the ADAPT condition, we obtain
    \begin{equation*}
        \|\phi(u_{\theta^{k}})-\phi_k(u_{\theta^{k}})\|_V=\eta_{k}(u_{\theta^{k}}) \leq \eta_k(u_{\theta^{k-1}}) \leq \epsilon_k.
    \end{equation*}
    Furthermore, the LEARN condition ensures that
    \begin{equation*}
        \|\phi_k(u_{\theta^{k}})\|_V \leq \epsilon_k.
    \end{equation*}
    Applying these two inequalities to the error bound in Equation \eqref{Eq: true error upper bound} with $u_\theta = u_{\theta^{k}}$ yields
    \begin{equation*}
    \|u - u_{\theta^k}\|_U \leq \frac{1}{\alpha} (\epsilon_k + \epsilon_k) = \frac{2}{\alpha} \delta^k\epsilon_0,
    \end{equation*}
    which completes the proof.
\end{proof}

In practice, training a neural network using a sufficiently fine mesh drives the approximation toward a minimizer of the variational residual rather than the exact solution $u$. 
Therefore, we analyze the error of the trained neural network by considering the best-in-class approximation, which minimizes the continuous residual norm.
Since the set of neural networks $U_{NN}$ is neither a linear, closed, nor convex subset of $U$, a unique minimizer of $\|\phi(u_\theta)\|_V$ may not exist within $U_{NN}$ \cite{petersen2021topological}. To account for this, we evaluate the approximation error by considering a quasi-minimizer (cf. \cite{brevis2022neural}).
\begin{definition}
    Let $\delta_{NN}>0$ be a given tolerance. A neural network $u_{\tilde\theta}\in U_{NN}$ is called a quasi-minimizer if it satisfies:
    \begin{align}\label{minimizer of phi}
        \|\phi(u_{\tilde\theta})\|_V \leq  \inf_{\theta \in \mathbb{R}^s} \|\phi(u_{\theta})\|_V + \delta_{NN}.
    \end{align}
\end{definition}

\begin{theorem}\label{Thm: norm of nn error}
Let $u_{\tilde\theta}$ be a quasi-minimizer defined in \eqref{minimizer of phi}, and let $u$ be the exact solution to Problem \eqref{eq:weakPDE}. Assume that $u_{\theta^{k}} \in U_{NN}$ is the neural network obtained at the $k$th iteration of the sequential algorithm and $\eta_{k}(u_{\theta^{k}}) \leq \eta_{k}(u_{\theta^{k-1}})$. Then, we have
\begin{align*}
    \|u_{\tilde\theta} - u_{\theta^{k}}\|_U \leq \dfrac{1}{\alpha}\left(\mu\inf_{\theta \in \mathbb{R}^s}\|u-u_\theta\|_U + \delta_{NN} + 2\delta^k\epsilon_{0}\right),
\end{align*}
where $\mu$ and $\alpha$ are the constants from Equations \eqref{eq:bound_cont} and \eqref{eq:inf_sup_cont}, respectively.
\end{theorem}

\begin{proof}
    Starting from the adjoint inf-sup condition in Equation \eqref{eq:inf_sup_cont_dual}, utilizing the linearity of the bilinear form $b(\cdot,\cdot)$, and using the definition of the Riesz representatives, we have:
    \begin{align*}
        \|u_{\tilde\theta} - u_{\theta^{k}}\|_U &\leq \dfrac{1}{\alpha}\sup\limits_{0\neq v \in V}\dfrac{b(u_{\tilde\theta} - u_{\theta^{k}},v)}{\|v\|_V} 
        = \dfrac{1}{\alpha}\sup\limits_{0\neq v \in V}\dfrac{r(u_{\tilde\theta},v) - r(u_{\theta^{k}},v)}{\|v\|_V} \\
        &= \dfrac{1}{\alpha}\sup\limits_{0\neq v \in V}\dfrac{(\phi(u_{\tilde\theta})-\phi(u_{\theta^{k}}),v)_V}{\|v\|_V} 
        = \dfrac{1}{\alpha}\|\phi(u_{\tilde\theta})-\phi(u_{\theta^{k}})\|_V. 
    \end{align*}
    Applying the triangle inequality, we obtain:
    \begin{align*}
        \|u_{\tilde\theta} - u_{\theta^{k}}\|_U \leq \dfrac{1}{\alpha}\left(\|\phi(u_{\tilde\theta})\|_V + \|\phi(u_{\theta^{k}})-\phi_k(u_{\theta^{k}})\|_V + \|\phi_k(u_{\theta^{k}})\|_V\right).
    \end{align*}
    Using the definition of a quasi-minimizer \eqref{minimizer of phi} and substituting the LEARN and ADAPT conditions, we obtain:
    \begin{align*}
        \|u_{\tilde\theta} - u_{\theta^{k}}\|_U \leq \dfrac{1}{\alpha}\left(\inf_{\theta \in \mathbb{R}^s}\|\phi(u_{\theta})\|_V + \delta_{NN} + 2\epsilon_{k}\right).
    \end{align*}
    Finally, applying Equation \eqref{eq:robust estimation} completes the proof of this theorem.
\end{proof}

\begin{remark}
    Theorem \ref{Thm: norm of nn error} demonstrates that the sequential algorithm can guide the neural network toward a quasi-minimizer when the neural network architecture is sufficiently deep or wide to ensure that the error of the best approximation is very small.
\end{remark}

\subsection{Refinement indicator}\label{sec: refinement indicator}

We note that standard RVPINNs might reach a performance plateau because, after some training iterations, the discrete Riesz representative in the finite-dimensional test space $V_k$ fails to adequately approximate the exact Riesz representative in the infinite-dimensional test space $V$. 
To make the discrepancy $\eta_k(u_\theta)=\|\phi(u_\theta)-\phi_k(u_\theta)\|_V$ small, the discrete test space $V_k$ must be sufficiently rich. However, training a neural network using an excessively refined test space from the beginning is computationally expensive. While one could use uniform refinement once the loss reaches a target level, a more efficient approach is adaptive refinement. By employing an appropriate refinement indicator, we can enrich the test space only where necessary, maintaining accuracy without incurring unnecessary computational costs. 

We define a computable refinement indicator that acts as a surrogate for the discrepancy between the continuous and discrete Riesz representatives. 
Specifically, we introduce a two-level refinement criterion that compares Riesz representatives from two discrete test spaces. Let $V_k \subset \hat V_k \subset V$ be two nested discrete spaces, where $V_k$ is the original test space (based on some polynomial order) and $\hat{V}_k$ is an enriched space (for example, by using higher-order polynomials or a finer subgrid). We denote the discrete Riesz representatives associated with $V_k$ and $\hat{V}_k$ as $\phi_k(u_{\theta})$ and $\hat\phi_k(u_{\theta})$, respectively. Let $\mathcal{T}_k$ be the (triangular) mesh associated with these spaces, and let $T\in \mathcal{T}_k$ denote an element (triangle) in the mesh. For each element $T$, we compute the local discrepancy indicator:
    \begin{equation}\label{Eq: two-grid indicator}
        \iota_k^2(T;u_\theta) := \|\hat\phi_k(u_{\theta}) - \phi_k(u_{\theta})\|_{V(T)}^2.
    \end{equation} 
This indicator measures the local difference between the two discrete Riesz representatives and identifies regions where the test space $V_k$ needs further enrichment. The global two-level indicator is defined by:
\begin{equation}\label{Eq: global estimator}
    \iota_k(\mathcal{T}_k;u_\theta):=\|\hat\phi_k(u_\theta) - \phi_k(u_\theta)\|_V=\sqrt{\sum_{T \in \mathcal{T}_k}\iota_k^2(T;u_\theta)}.
\end{equation}
The following theorem shows that the global indicator is equivalent to the true discrepancy 
$\|\phi(u_\theta) - \phi_k(u_\theta)\|_V$. This holds under a standard saturation assumption \cite[Section 5.2]{ainsworth2000posteriori}, which assumes that the enriched Riesz representative  
$\hat{\phi}_k(u_\theta)$ is a better approximation of the continuous representative $\phi(u_\theta)$ than $\phi_k(u_\theta)$. 
This equivalence establishes the reliability and efficiency of $\iota_k$ as a refinement indicator for the adaptive process.

\begin{theorem}[Equivalence of the two-level indicator]
Suppose the saturation assumption holds: there exists a constant $\beta \in [0,1)$ such that 
$$\|\phi(u_\theta) - \hat\phi_{k}(u_\theta) \|_{V} \leq \beta \| \phi(u_\theta) - \phi_{k}(u_\theta) \|_{V}.$$
Then, the global refinement indicator satisfies the following two-sided bound:
$$ \|\hat \phi_k(u_\theta) - \phi_k(u_\theta) \|_{V} \leq \|\phi(u_\theta) - \phi_k(u_\theta) \|_{V} \leq \frac{1}{\sqrt{1 - \beta^2}}\|\hat \phi_k(u_\theta) - \phi_k(u_\theta) \|_{V}.$$
\end{theorem}
\begin{proof}
    Lower bound: Using Galerkin orthogonality $(\hat\phi_k(u_\theta) - \phi(u_\theta), v)_V = 0$ for all \(v \in \hat V_k\), and the Cauchy-Schwarz inequality, we have
\begin{align*}
    \nonumber
    \| \hat\phi_k(u_\theta) - \phi_k(u_\theta) \|_{V}^2 
    &= (\hat\phi_k(u_\theta) - \phi_k(u_\theta), \hat\phi_k(u_\theta) - \phi_k(u_\theta))_V \\
    &= (\hat\phi_k(u_\theta) - \phi(u_\theta), \hat\phi_k(u_\theta) - \phi_k(u_\theta))_V \\
    &\qquad + (\phi(u_\theta) - \phi_k(u_\theta), \hat\phi_k(u_\theta) - \phi_k(u_\theta))_V \\
    &\leq 0 + \|\phi(u_\theta) - \phi_k(u_\theta) \|_{V}\|\hat\phi_k(u_\theta) - \phi_k(u_\theta) \|_{V}.
\end{align*}
By canceling, we obtain the desired lower bound.

Upper bound: Expanding and using Galerkin orthogonality, we have
\begin{align*}
\|\phi(u_\theta) - \phi_k(u_\theta)\|_{V}^2 & = \|\phi(u_\theta) - \hat\phi_k(u_\theta)\|^2_{V} + 2(\phi(u_\theta) - \hat\phi_k(u_\theta), \hat\phi_k(u_\theta) - \phi_k(u_\theta))_V \\
&\qquad + \|\hat\phi_k(u_\theta) - \phi_k(u_\theta)\|^2_{V} \\
& = \|\phi(u_\theta) - \hat\phi_k(u_\theta)\|^2_{V} + \|\hat\phi_k(u_\theta) - \phi_k(u_\theta)\|^2_{V}.
\end{align*}
Applying the saturation assumption gives
\begin{equation*}
   \|\phi(u_\theta) - \phi_k(u_\theta)\|_{V}^2 \leq \beta^2 \|\phi(u_\theta) - \phi_{k}(u_\theta)\|^2_{V} + \|\hat\phi_k(u_\theta) - \phi_k(u_\theta)\|^2_{V}.
\end{equation*}
Rearranging yields $\sqrt{1 - \beta^2}\|\phi(u_\theta) - \phi_{k}(u_\theta)\|_{V} \leq \|\hat\phi_k(u_\theta) - \phi_k(u_\theta)\|_{V}$
which completes the proof.
\end{proof}

\begin{remark}
    One could use an a posteriori error estimator for $\|\phi(u_{\theta}) - \phi_k(u_{\theta})\|_{V}^2$ as a refinement indicator. In \cite{fuhrer2025posteriori}, this estimator is an additional term in the RVPINN loss function.
\end{remark}

\section{Numerical results}\label{sec:results}

In this section, we provide a practical algorithm for the adaptivity strategy in Section \ref{sec: post-training adap}. We demonstrate the performance of the proposed algorithm on three benchmark problems arising from the model elliptic problem introduced in Section \ref{sec: model}: a smooth Poisson problem, an elliptic interface problem with a kink solution, and a Poisson problem on an L-shaped domain with a corner singularity.

\subsection{Practical adaptivity strategy}\label{sec:algorithm}

The practical adaptive procedure presented here is based on the standard SOLVE--ESTIMATE--MARK--REFINE structure in the classical adaptive finite element method (AFEM) loop \cite{cohen2012convergence}. We integrate these processes into the sequential adaptivity strategy in Section \ref{sec: post-training adap}. We use $\iota_k(\mathcal{T}_k;u_\theta)$ as a computable error estimator for $\eta_k(u_\theta)$. We assume that the neural network architecture has sufficient expressive power to approximate the solution and that the nonlinear optimizer (e.g., Adam) effectively minimizes the loss function.

Let $\mathcal{T}_0$ be an initial mesh of the domain, and let $\mathcal{T}_k$ denote the mesh after $k$ refinement steps. We define $V_k$ as the discrete test space associated with $\mathcal{T}_k$, typically constructed using low-order polynomials. Let $\mathcal{L}_k$ be the corresponding loss function, and let $u_{\theta^{k}}$ be the neural network approximation associated with $\mathcal{L}_k$. Given a starting tolerance $\epsilon_0 >0$, a scaling factor $0 < \delta <1$, and a marking parameter $0 < \gamma \leq 1$, we initially train the neural network to obtain $u_{\theta^0}$ such that $\sqrt{\mathcal{L}_0(u_{\theta^0})}\leq \epsilon_0$. 
The algorithm for the practical strategy is described in Algorithm \ref{Algorithm: Adaptive RVPINNs}.

\begin{algorithm}[h!]
\begin{algorithmic}[1]
\REQUIRE $u_{\theta^0}$, $\mathcal{T}_0$, $\mathcal{L}_0$, $\epsilon_0$, $\delta, \gamma$, $K$ \smallskip
\FOR{k = 1 \textbf{to} K}
\STATE \COMMENT{\textbf{--- ADAPT (Test mesh refinement) ---}}
\WHILE{$\iota_{k-1}(\mathcal{T}_{k-1};u_{\theta^{k-1}}) > \delta^{k} \epsilon_0$} \smallskip
    \STATE $\{\iota_{k-1}^2(T;u_{\theta^{k-1}})\}_{T \in \mathcal{T}_{k-1}} = \mathrm{ESTIMATE}(\mathcal{T}_{k-1},u_{\theta^{k-1}})$
    \STATE $\mathcal{M}_{k-1} = \mathrm{MARK}(\{\iota_{k-1}^2(T;u_{\theta^{k-1}})\}_{T \in \mathcal{T}_{k-1}},\gamma)$
    \STATE $\mathcal{T}_{k-1} = \mathrm{REFINE}(\mathcal{T}_{k-1},\mathcal{M}_{k-1})$ \smallskip
\ENDWHILE \smallskip
    \STATE $\mathcal{T}_{k} \leftarrow \mathcal{T}_{k-1}$. \smallskip
    \STATE \COMMENT{\textbf{--- LEARN (Neural network optimization) ---}}
\WHILE{$\sqrt{\mathcal{L}_{k}(u_{\theta^{k-1}})} > \delta^{k}\epsilon_0$}       \smallskip
    \STATE $u_{\theta^{k-1}} = \mathrm{SOLVE}(\mathcal{T}_{k},u_{\theta^{k-1}})$ 
    \smallskip
\ENDWHILE \smallskip
\STATE $u_{\theta^{k}} \leftarrow u_{\theta^{k-1}}$.
\ENDFOR \smallskip
\RETURN Optimized neural network $u_{\theta^{K}}$.
\end{algorithmic}
\caption{Adaptive training algorithm.}\label{Algorithm: Adaptive RVPINNs}
\end{algorithm}

Procedure descriptions

\begin{itemize}  
    \item ESTIMATE: For a predefined test space $V_{k-1}$ (based on some polynomial order) on the current mesh $\mathcal{T}_{k-1}$, we construct an enriched test space $\hat V_{k-1}$ (typically by increasing the polynomial order) such that $V_{k-1} \subset \hat V_{k-1} \subset V$. We then compute the discrete Riesz representatives $\phi_{k-1}(u_{\theta^{k-1}})$ and $\hat\phi_{k-1}(u_{\theta^{k-1}})$ with respect to these spaces. For each element $T \in \mathcal{T}_{k-1}$, we calculate the local discrepancy indicator $\iota_{k-1}^2(T;u_{\theta^{k-1}})$ as defined in \eqref{Eq: two-grid indicator}. 
    
    \item MARK: We apply the D\"orfler marking criterion \cite{dorfler1996convergent} to identify the elements that account for the largest portion of the discrepancy. In particular, we find a smallest subset of elements $\mathcal{M}_{k-1} \subset \mathcal{T}_{k-1}$ such that: 
    \begin{equation}\label{Dorfler marking}
        \sum_{T \in \mathcal{M}_{k-1}}\iota_{k-1}^2(T;u_{\theta^{k-1}}) \geq \gamma \sum_{T \in \mathcal{T}_{k-1}}\iota_{k-1}^2(T;u_{\theta^{k-1}}).
    \end{equation}

    \item REFINE: We refine the marked elements in $\mathcal{M}_{k-1}$ using a conforming refinement scheme to generate a finer mesh $\mathcal{T}_{k-1}$. We then compute the global estimator $\iota_{k-1}(\mathcal{T}_{k-1};u_{\theta^{k-1}})$ in \eqref{Eq: global estimator}. If $\iota_{k-1}(\mathcal{T}_{k-1};u_{\theta^{k-1}}) > \delta^{k}\epsilon_0$, we repeat all procedures in ADAPT to further enrich the test space. Otherwise, this mesh defines an enriched test space mesh $\mathcal{T}_{k}$ for the next level. 

    \item SOLVE: We train the neural network on the current mesh $\mathcal{T}_{k}$ by minimizing the loss function $\mathcal{L}_{k}$ using Adam optimizer \cite{kingma2014adam}. This process is repeated until we obtain the neural network $u_{\theta^{k}}$ such that $\sqrt{\mathcal{L}_{k}(u_{\theta^{k}})}\leq\delta^{k}\epsilon_0$.
\end{itemize}

In the following sections, we present the numerical results for Algorithm \ref{Algorithm: Adaptive RVPINNs}. We analyze the convergence rates of our approach and compare them with those of the standard Galerkin FEM using identical trial and test spaces and strong imposition of Dirichlet boundary conditions (hereafter referred to simply as FEM). The comparison is made with respect to the test-space dimension, not total computational cost.
All reported true errors are computed using a four-point quadrature rule per element on a highly refined uniform mesh (consisting of $8192$ elements for square domains and $6144$ elements for the L-shaped domain).

\subsection{Smooth solution}\label{subsec: smooth}

We first consider the Poisson problem on $\Omega = (0,1)^2$ with homogeneous Dirichlet boundary condition: find $u$ such that
\begin{equation}\label{Poisson smooth}
    \begin{cases}
        \begin{aligned}
        -\Delta u &= f &&\textrm{in } \Omega, \\
        u &= 0 &&\textrm{on } \Gamma,
        \end{aligned}
    \end{cases} 
\end{equation}
where $f(x,y) = 2\pi^2\sin(\pi x)\sin(\pi y) \in L^2(\Omega)$. This problem admits a smooth solution $u(x,y) = \sin(\pi x)\sin(\pi y)$. The variational formulation of \eqref{Poisson smooth} is to find $u \in H_0^1(\Omega)$ such that
\begin{equation}\label{Poisson smooth weak}
    r(u,v) := (f,v)_{L^2(\Omega)} - (\nabla u, \nabla v)_{L^2(\Omega;\mathbb{R}^2)}  = 0 \quad \forall v \in H_0^1(\Omega).
\end{equation}

We approximate the solution $u$ using a neural network $u_{\theta}(x,y)=x(1-x)y(1-y)\bar{u}_\theta(x,y)$, where $\bar{u}_{\theta}$ is a six-layer, fully connected architecture with hyperbolic tangent activation functions, where each hidden layer contains $64$ neurons. This ensures that the neural network vanishes on the boundary. The domain $\Omega$ is initially partitioned into $32$ uniform triangular elements to form the starting mesh $\mathcal{T}_0$ as shown in Figure \ref{Figure: mesh_square} (a). For the test space, we choose a discrete space $V_k := \mathbb{P}_0^1(\mathcal{T}_k) \subset H_0^1(\Omega)$, equipped with the energy inner product \eqref{eq:weighted_energy_inner_product}. Since $a \equiv 1$ in this example, this coincides with the standard $H_0^1$ inner product. The corresponding loss function is defined as in \eqref{eq:AVPINNs_loss_mat}.

To implement Algorithm \ref{Algorithm: Adaptive RVPINNs}, we set $\epsilon_0 = 0.5, \delta = 0.95$, and $K=120$. We compute integrals using a four-point quadrature rule per element. For training, we employ the Adam optimizer as a nonlinear solver with an adaptive learning rate, initialized at $0.0005$ and decayed by a factor of $0.9$ every $1{,}000$ iterations. We employ the refinement indicator defined in \eqref{Eq: two-grid indicator} and choose the enriched test space as $\hat{V}_k:=\mathbb{P}_0^2(\mathcal{T}_k)$. We refine all elements marked by the D\"orfler marking criterion with $\gamma = 0.2$.

\begin{figure}[t]
    \centering
    \subfloat[Initial mesh]{\includegraphics[width=0.3\linewidth]{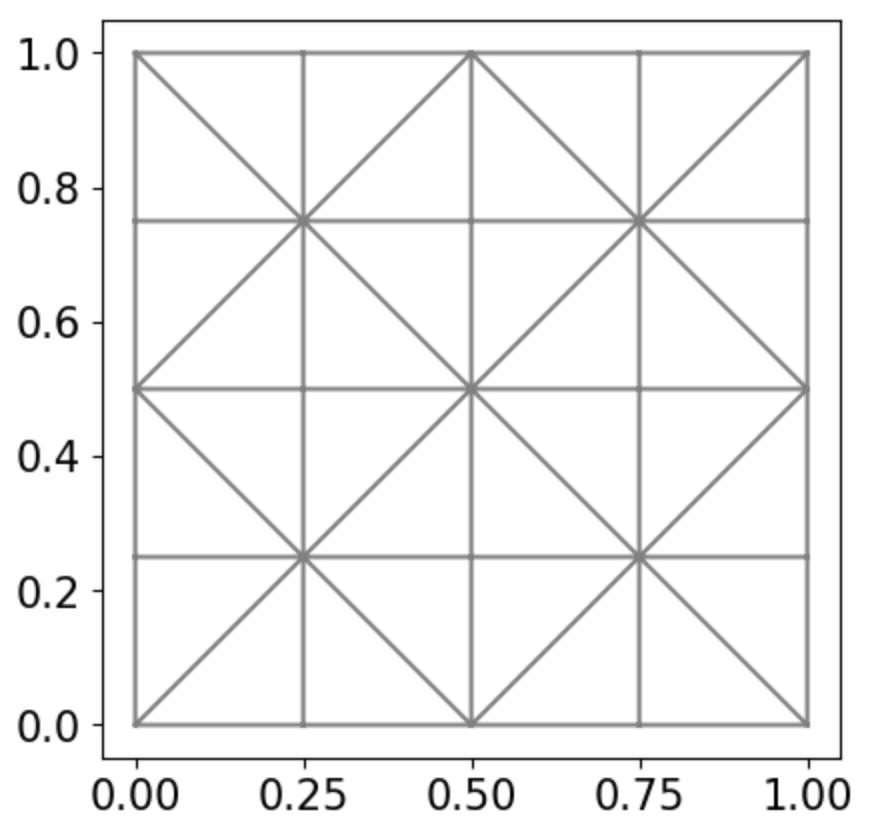}}\qquad 
    \subfloat[Refined mesh]{\includegraphics[width=0.3\linewidth]{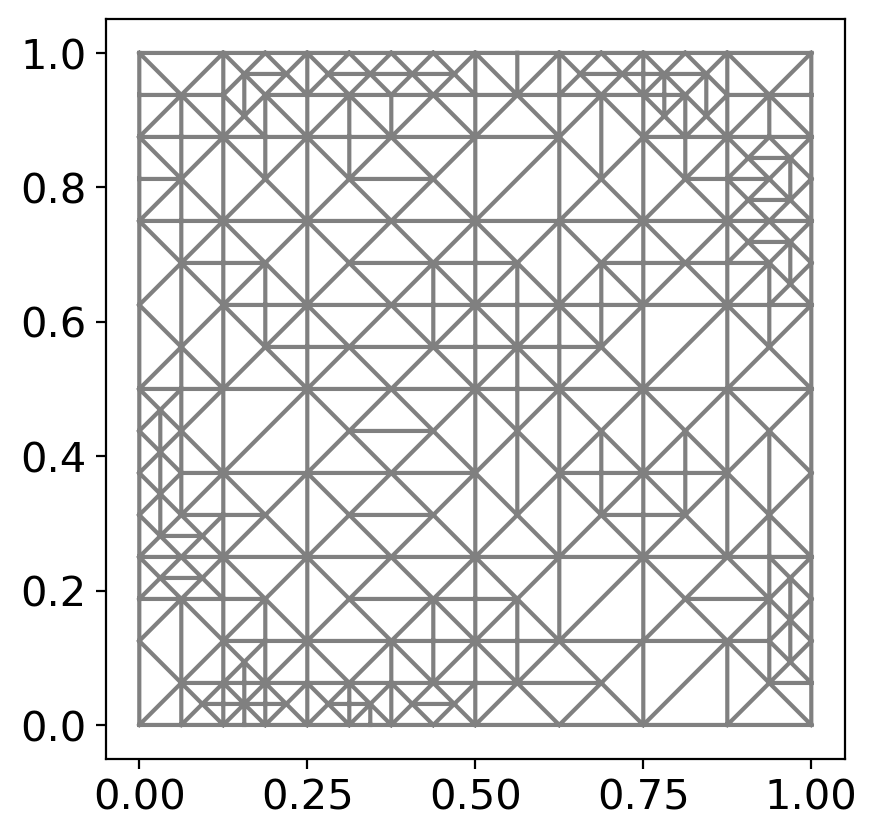}}
    \caption{Test-space meshes for the smooth solution.}
    \label{Figure: mesh_square}
\end{figure}

\begin{figure}[!]
    \subfloat[$H^1$-error and square root of the loss vs. epochs]{\includegraphics[width=0.32\linewidth]{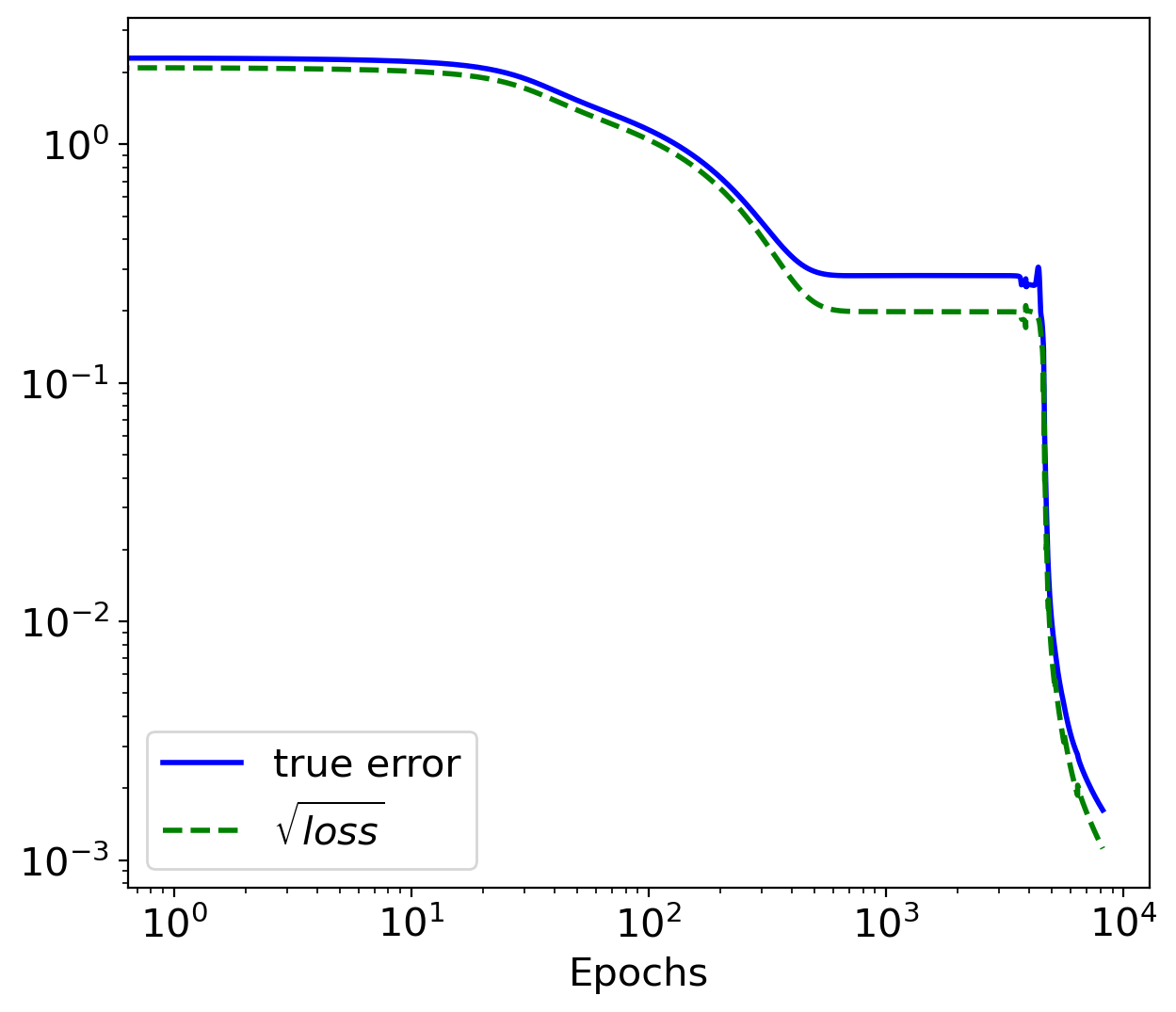}} ~
    \subfloat[$H^1$-error vs. test space dimension]{\includegraphics[width=0.33\linewidth]{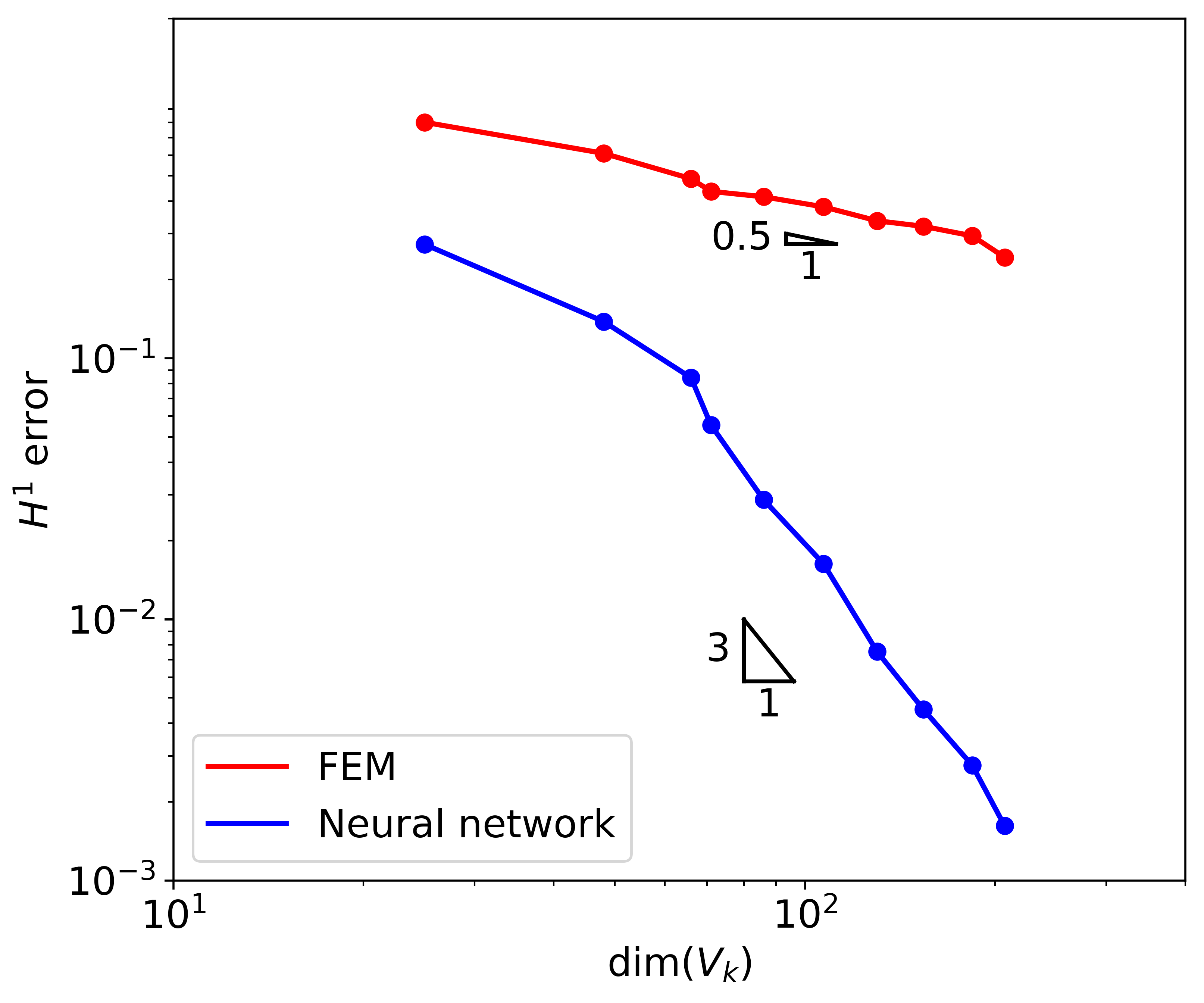}} ~
    \subfloat[$H^1$-error, $\sqrt{\mathrm{loss}}$, and refinement indicator vs. test space dimension]{\includegraphics[width=0.32\linewidth]{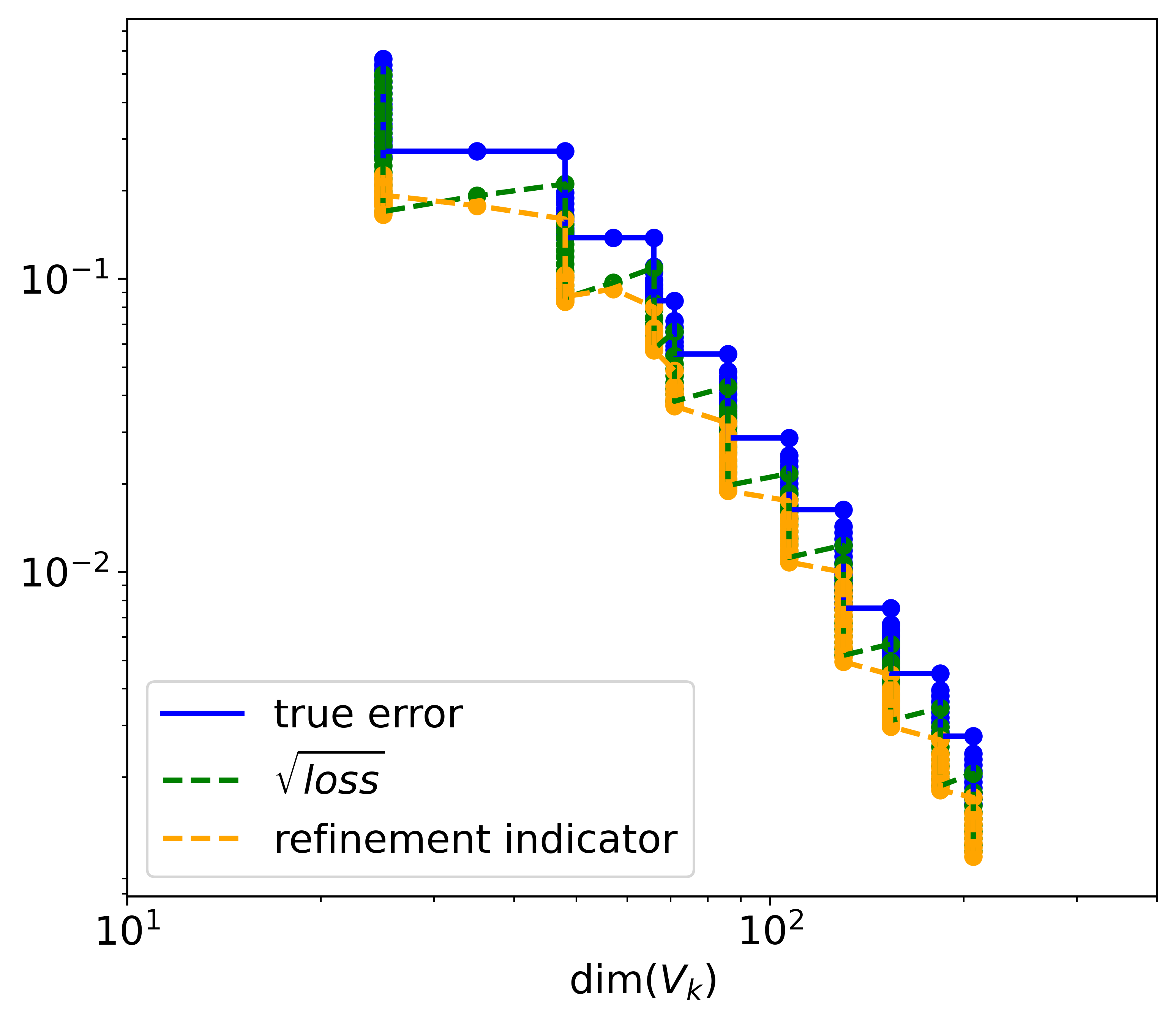}}
    \caption{Results for the smooth solution.}
    \label{Figure: square Poisson}
\end{figure}

Figure \ref{Figure: mesh_square} (b) illustrates the refined mesh, consisting of $448$ elements. Figure \ref{Figure: square Poisson} (a) demonstrates that the $H^1$-error, $\|u-u_{\theta}\|_{H^1(\Omega)}$, and the square root of the loss coincide throughout the training process. This alignment confirms that controlling the discrepancy between two Riesz representatives ensures training robustness. Furthermore, we evaluate the performance of the adaptive RVPINN algorithm by comparing it to a FEM implementation utilizing the same discrete test space and mesh configurations. This comparison is based on the $H^1$-norm of the true error versus the test-space dimension. In Figure \ref{Figure: square Poisson} (b), the convergence rate of the neural network approximation for this example is approximately $3$, exhibiting a faster empirical decay than the reference piecewise-linear FEM solution rate of $0.5$ (cf. \cite{cascon2008quasi}) when both errors are plotted against the test-space dimension. This reflects the high expressivity of the fixed neural network architecture, although this comparison does not account for the overall computational cost.

Figure \ref{Figure: square Poisson} (c) illustrates the relationship between the square root of the loss function and the refinement indicator as they reach the specified tolerances $\delta^k\epsilon_0$. While mesh refinement initially increases the loss due to the expansion of the discrete test space, it simultaneously reduces the refinement indicator. Conversely, training the neural network minimizes the loss and generally tends to reduce the refinement indicator, although this reduction is not strictly guaranteed. Empirically, training the neural network using Algorithm \ref{Algorithm: Adaptive RVPINNs} ensures a consistent reduction in the true error.

\subsection{Kink solution}

We consider the elliptic interface problem on $\Omega = (-1,1)^2$: find $u$ such that
\begin{equation}\label{elliptic interface}
    \begin{cases}
        \begin{aligned}
        -\nabla\cdot (a\nabla  u) &= f &&\textrm{in } \Omega, \\
        u &= g &&\textrm{on } \Gamma,
        \end{aligned}
    \end{cases} 
\end{equation}
where $a$ is a piecewise-constant diffusion coefficient. In polar coordinates $(\rho,\psi)$, $a$ is defined as
\begin{equation*}
    a(\rho,\psi) = \begin{cases}
        a_1, & \rho < \rho_0 \\
        a_2, & \rho \geq \rho_0,
    \end{cases}
\end{equation*}
with $\rho=\sqrt{x^2+y^2}$ and $\rho_0 = \pi/6.28$.
The functions $f$ and $g$ are given in polar coordinates by $f(\rho,\psi) = -9\rho$ and
\begin{equation}\label{eq: g interface}
    g(\rho,\psi)=\frac{\rho^3}{a_2}+\left(\frac{1}{a_1}-\frac{1}{a_2}\right)\rho_0^3.
\end{equation}
The exact solution to this problem \cite{wang2020mesh} is
\begin{equation}\label{Eq: sol interface}
    u(\rho,\psi) = \begin{cases}
        \frac{\rho^3}{a_1}, & \rho < \rho_0 \\
        \frac{\rho^3}{a_2}+\left(\frac{1}{a_1}-\frac{1}{a_2}\right)\rho_0^3, & \rho \geq \rho_0.
    \end{cases}
\end{equation}
In our experiment, we set $a_1 = 1$ and $a_2 = 5$.
The variational formulation of \eqref{elliptic interface} is to find $u \in H^1(\Omega)$ such that $u=g$ on $\Gamma$ and
\begin{equation}\label{elliptic interface weak}
    r(u,v) := (f,v)_{L^2(\Omega)} - (a\nabla u, \nabla v)_{L^2(\Omega;\mathbb{R}^2)}  = 0 \quad \forall v \in H_0^1(\Omega).
\end{equation}

\begin{figure}[t]
    \subfloat[Initial mesh]{\includegraphics[width=0.3\linewidth]{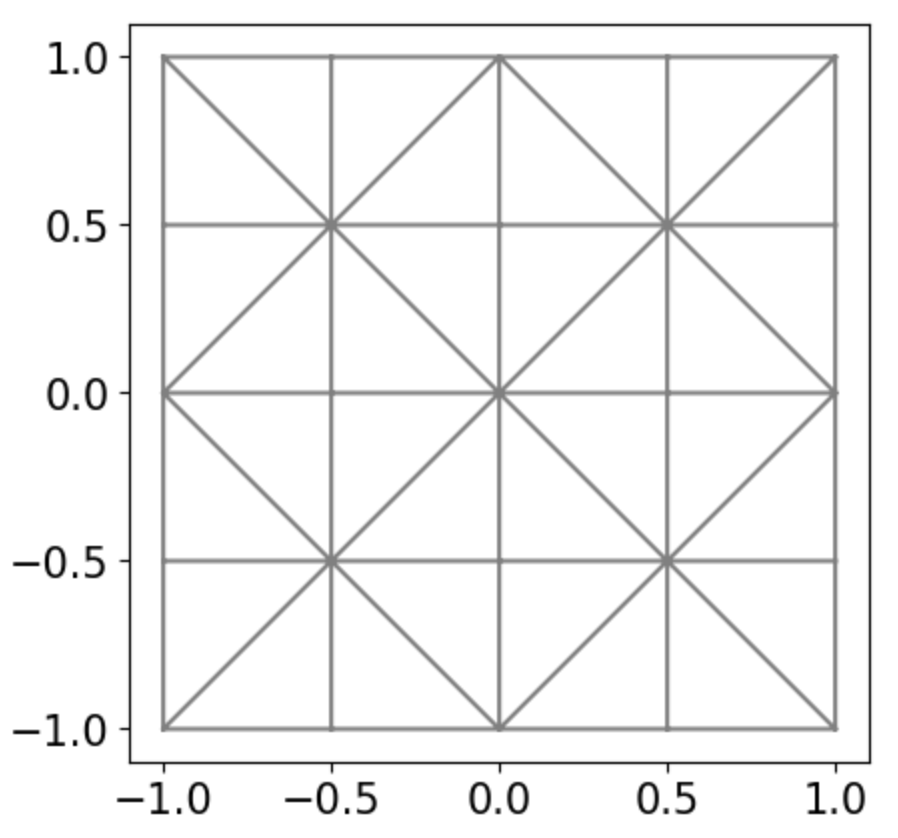}}~
    \subfloat[Refined mesh]{\includegraphics[width=0.3\linewidth]{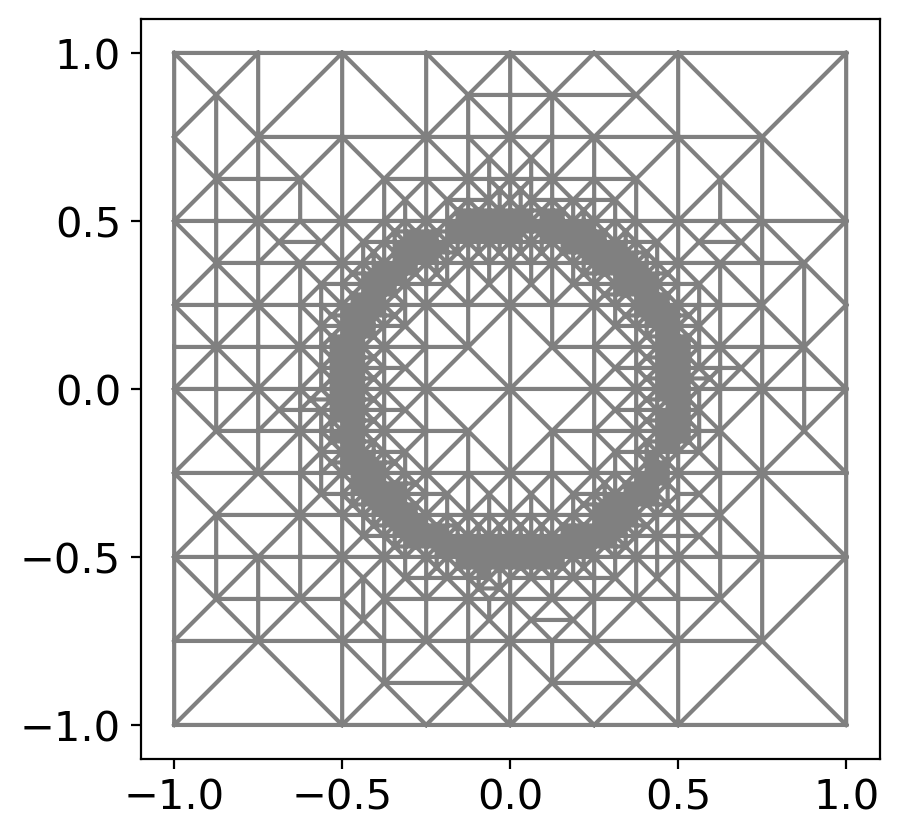}}~
    \subfloat[Pointwise absolute error]{\includegraphics[width=0.35\linewidth]{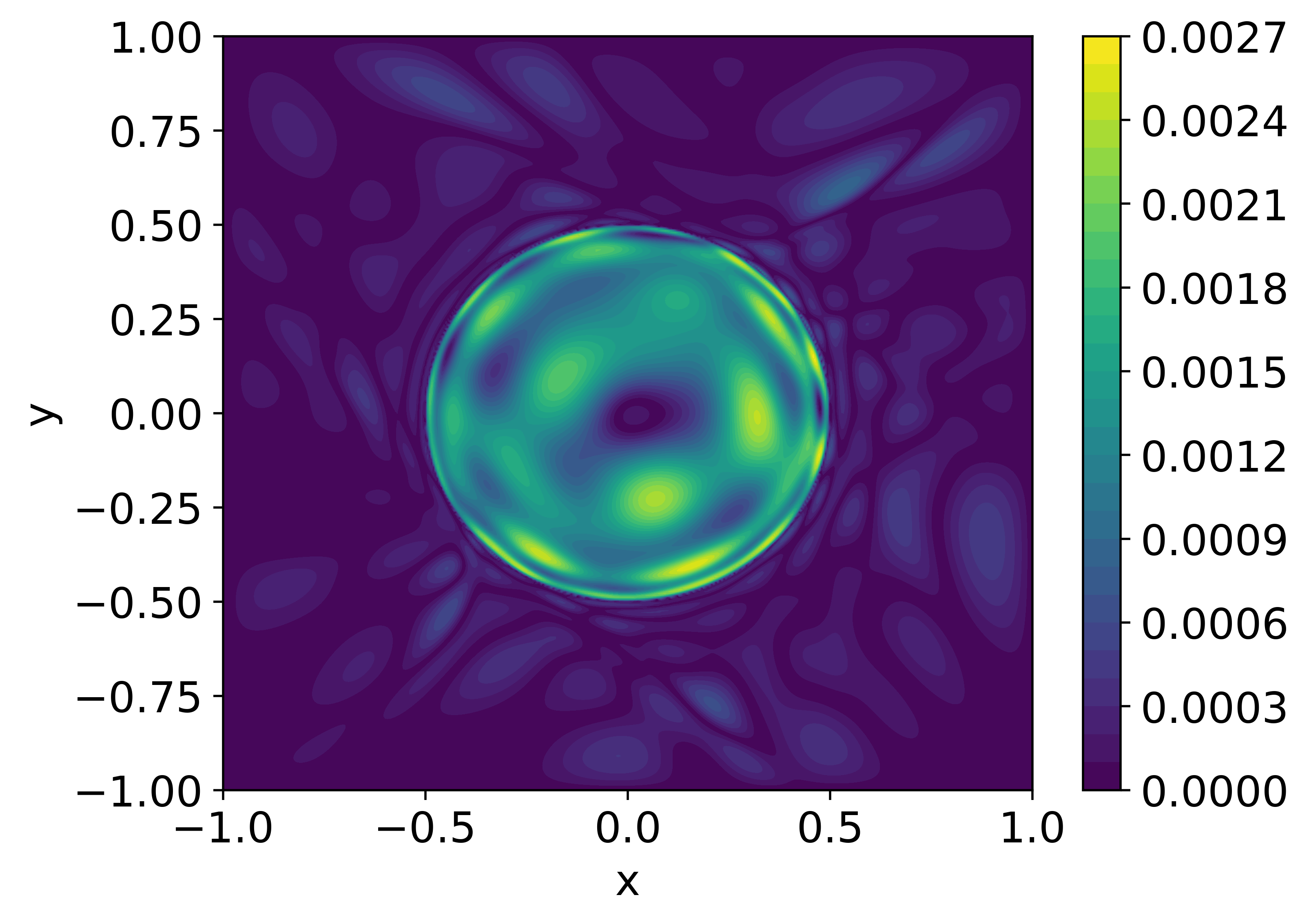}}
    \caption{Test-space mesh and pointwise absolute error for the kink solution.}
    \label{Figure: mesh kink}
\end{figure}

We approximate the solution $u$ using a neural network $u_{\theta}(x,y)=(1-x^2)(1-y^2)\bar{u}_\theta(x,y)+g(x,y)$, where $\bar{u}_{\theta}$ is a six-layer, fully connected architecture with hyperbolic tangent activation functions, where each hidden layer contains $64$ neurons, and $g(x,y)$ is the Cartesian-coordinate function of the polar-coordinate function in \eqref{eq: g interface}. This ensures that the boundary condition is strongly imposed on the neural network structure. The domain $\Omega$ is initially partitioned into $32$ uniform triangular elements to form the starting mesh $\mathcal{T}_0$ as shown in Figure \ref{Figure: mesh kink} (a). For the test space, we choose a discrete space $V_k:=\mathbb{P}_0^1(\mathcal{T}_k) \subset H_0^1(\Omega)$, equipped with the weighted energy inner product \eqref{eq:weighted_energy_inner_product}. The corresponding loss function is defined as in \eqref{eq:AVPINNs_loss_mat}. 
We set $K=50$ and use the training process and adaptive criteria described in the previous section. 

Figure \ref{Figure: mesh kink} (b) illustrates the final mesh, which consists of $4750$ elements. As expected, the adaptive refinement is concentrated along the circular interface where the diffusion coefficient is discontinuous, and the solution lacks smoothness.
The pointwise absolute errors are shown in Figure \ref{Figure: mesh kink} (c). It is evident that the neural network approximates the solution with high precision, maintaining absolute errors below $3\times10^{-3}$ throughout the entire domain. 

\begin{figure}[t]
    \centering
    \subfloat[Energy-norm error and $\sqrt{\mathrm{loss}}$ vs. epochs]{\includegraphics[width=0.35\linewidth]{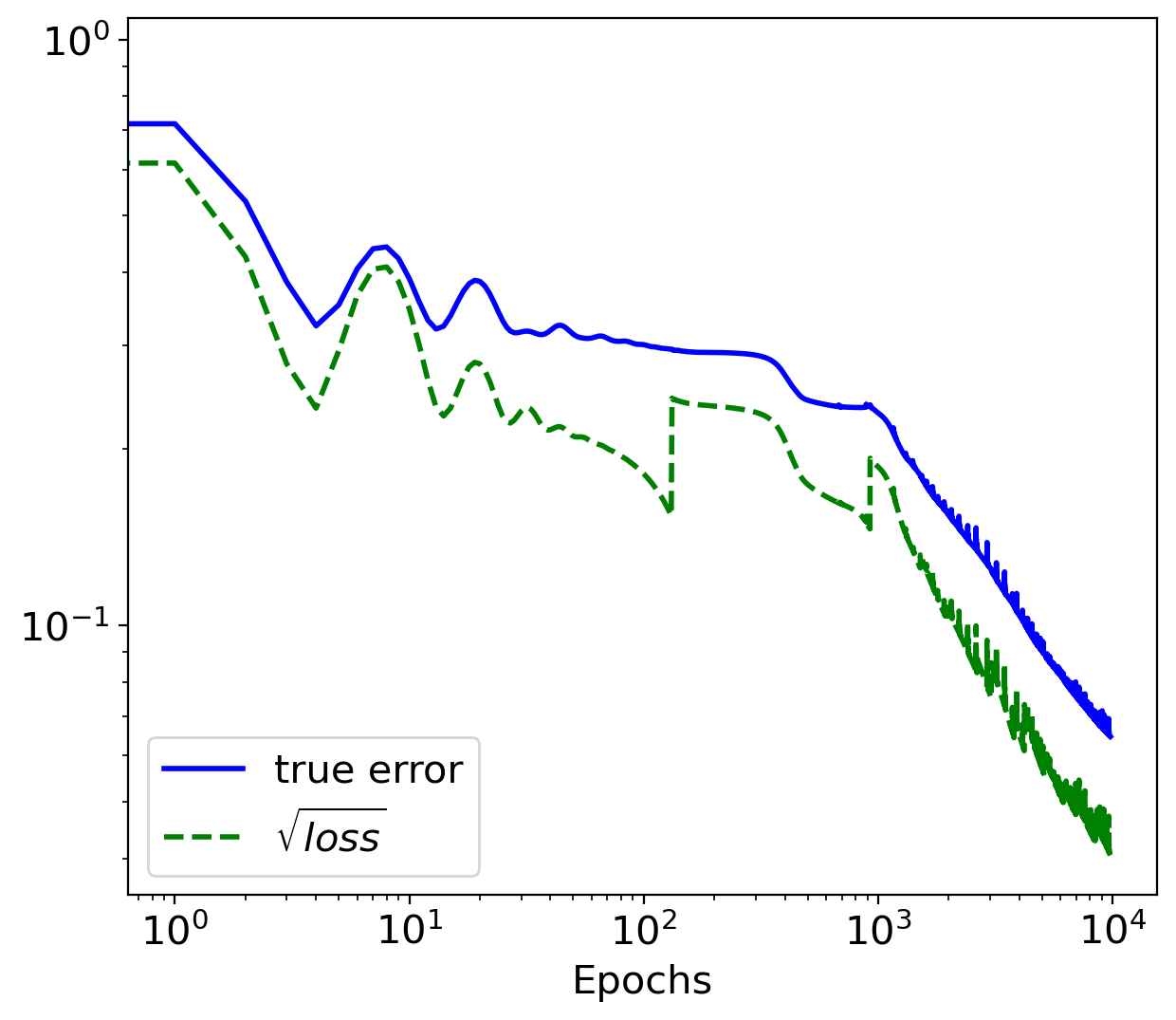}}\qquad 
    \subfloat[Energy-norm error vs. test space dimension]{\includegraphics[width=0.37\linewidth]{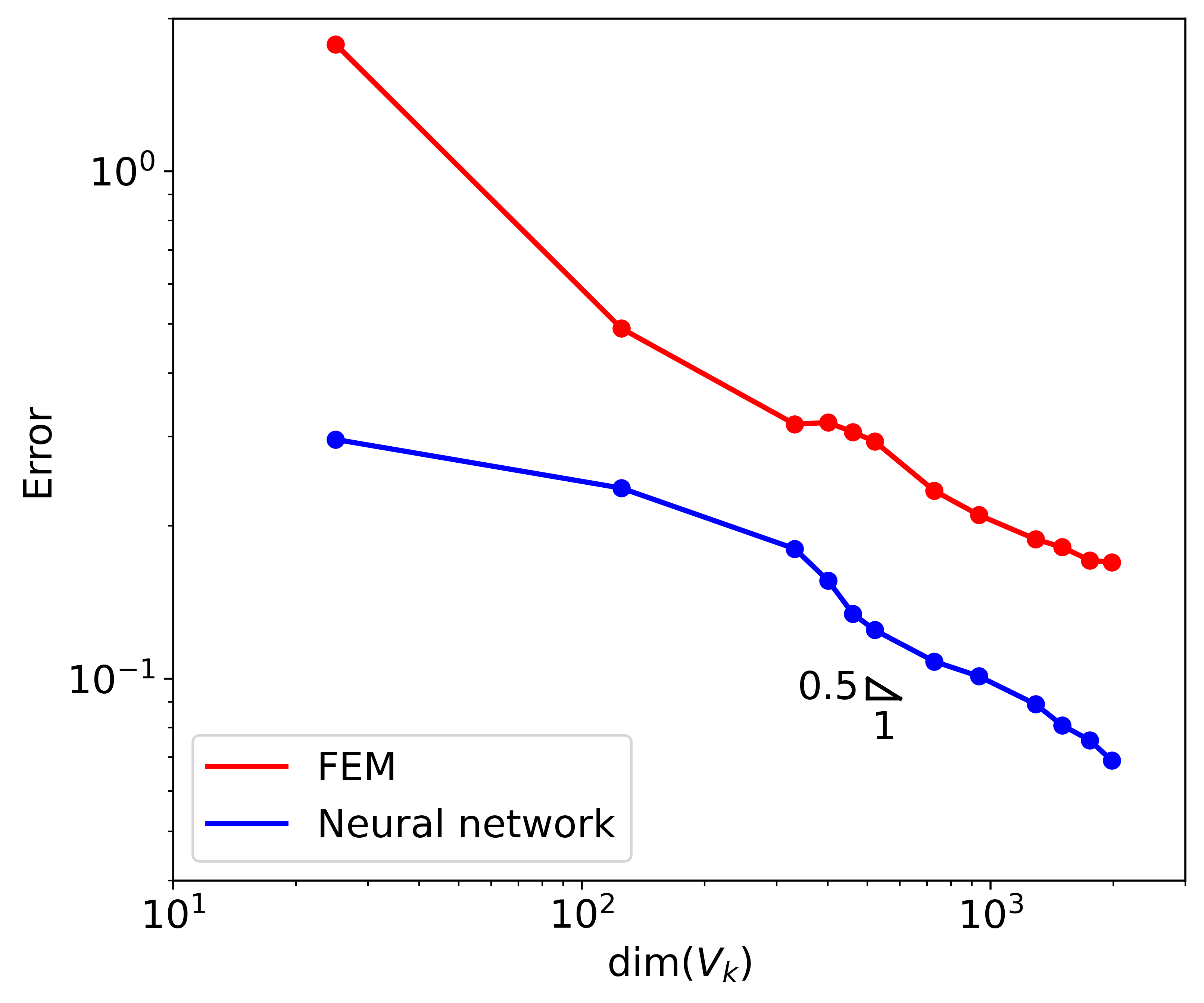}}
    \caption{Results for the kink solution.}
    \label{Figure: Poisson kink}
\end{figure}

Figure \ref{Figure: Poisson kink} (a) demonstrates that the error in the weighted energy norm, $\|a^{1/2}\nabla(u-u_{\theta})\|_{L^2(\Omega)}$, and the square root of the loss function remain generally parallel throughout the training process. As shown in Figure \ref{Figure: Poisson kink} (b), the observed convergence rate is approximately $0.5$, aligning with the optimal rate typically achieved by adaptive FEM. It highlights the performance of the proposed adaptive algorithm, which empirically achieves convergence rates that are competitive with, or superior to, those of standard adaptive FEM.

\subsection{Singular solution}

We consider the Poisson problem on the L-shaped domain $\Omega:=(-1,1)^2\backslash(-1,0]^2$:
find $u$ such that
\begin{equation}\label{Poisson problem}
    \begin{cases}
        \begin{aligned}
        -\Delta u &= 0 &&\textrm{in } \Omega, \\
        u &= g &&\textrm{on } \Gamma,
        \end{aligned}
    \end{cases} 
\end{equation}
where $g$ is defined in polar coordinates as $g(\rho,\psi)=\rho^{2/3}\sin(\frac{2}{3}(\pi-\psi))$.
The exact solution to this problem is $u(\rho,\psi) = g(\rho,\psi)$, which can be expressed in Cartesian coordinates as: 
\begin{equation}\label{Eq: L-shape sol}
    u(x,y) = \sqrt[3]{x^2+y^2}\sin\left(\frac{2}{3}(\pi-\mathrm{atan2}(y,x))\right),
\end{equation} 
where $\mathrm{atan2}$ denotes the four-quadrant inverse tangent. Note that the solution exhibits a corner singularity at the origin $(0,0)$, where the gradient becomes unbounded as $(x,y)\to(0,0)$, and belongs to the fractional Sobolev space $H^{5/3-\epsilon}(\Omega)$ for every $\epsilon > 0$ \cite{muga2023adaptive}.

The composite residual formulation of \eqref{Poisson problem} can be stated as follows: Find $u \in H^1(\Omega)$ such that 
\begin{equation}\label{eq: L-shape variational problem}
    \begin{aligned}
    r_\Omega(u,v) &:= -(\nabla u, \nabla v)_{L^2(\Omega;\mathbb{R}^2)}  = 0 &&\quad \forall v \in H_0^1(\Omega),\\
    r_\Gamma(u,\mathbf{v}) &:= \left<(g-u)|_{\Gamma}, \mathbf{v}\right>_{\Gamma} = 0 &&\quad \forall \mathbf{v} \in \mathbf{H}(\mathrm{div};\Omega).
    \end{aligned}
\end{equation}

We approximate the solution $u$ by a neural network $u_{\theta}$. The neural network $u_{\theta}$ is a six-layer, fully connected architecture with hyperbolic tangent activation functions, where each hidden layer contains $64$ neurons. The domain $\Omega$ is initially partitioned into $24$ uniform triangular elements to form the starting mesh $\mathcal{T}_0$ as shown in Figure \ref{Figure: mesh} (a). For the volume test space, we choose a discrete space $V_k^\Omega:=\mathbb{P}_0^1(\mathcal{T}_k) \subset H_0^1(\Omega)$, equipped with the energy inner product \eqref{eq:weighted_energy_inner_product}; here again $a \equiv 1$. For the boundary test space, we utilize the Raviart-Thomas space of order $1$, $V_k^\Gamma:=\mathbf{RT}^1(\mathcal{T}_k)$, which is a finite-dimensional subspace of $\mathbf{H}(\mathrm{div};\Omega)$. Since two variational formulations are considered, we define the loss function as in \eqref{Eq: loss two terms}. We set $K=70$ and use the training procedure described in the previous section. 

\begin{figure}[t]
    \centering
    \subfloat[Initial mesh]{\includegraphics[width=0.29\linewidth]{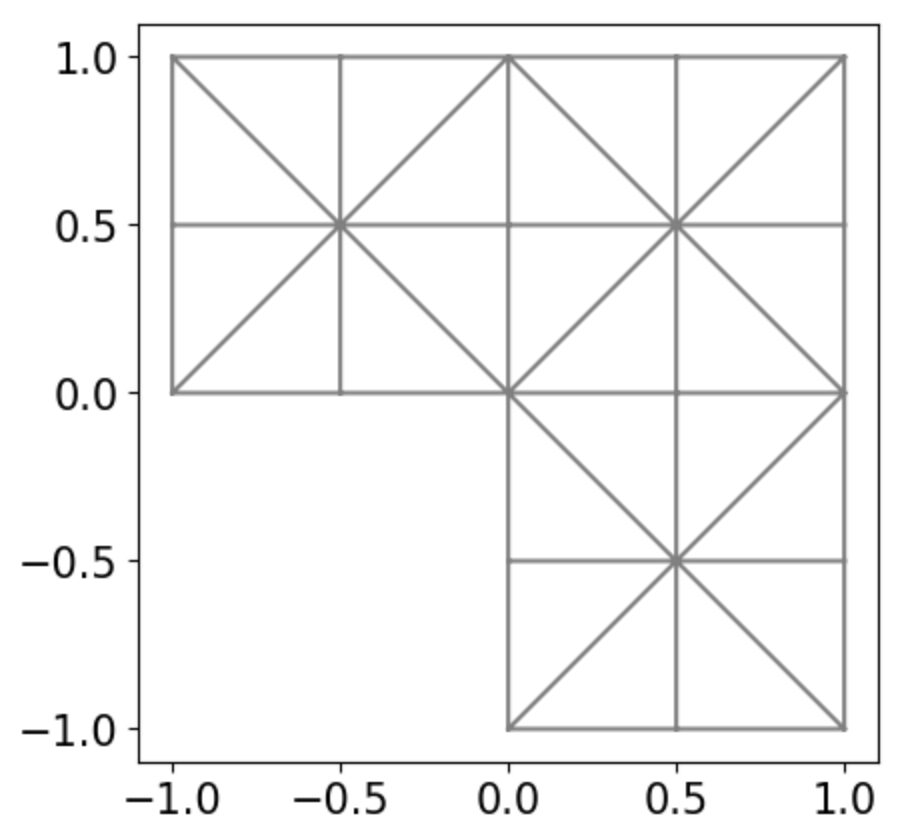}}~ 
    \subfloat[Refined mesh]{\includegraphics[width=0.29\linewidth]{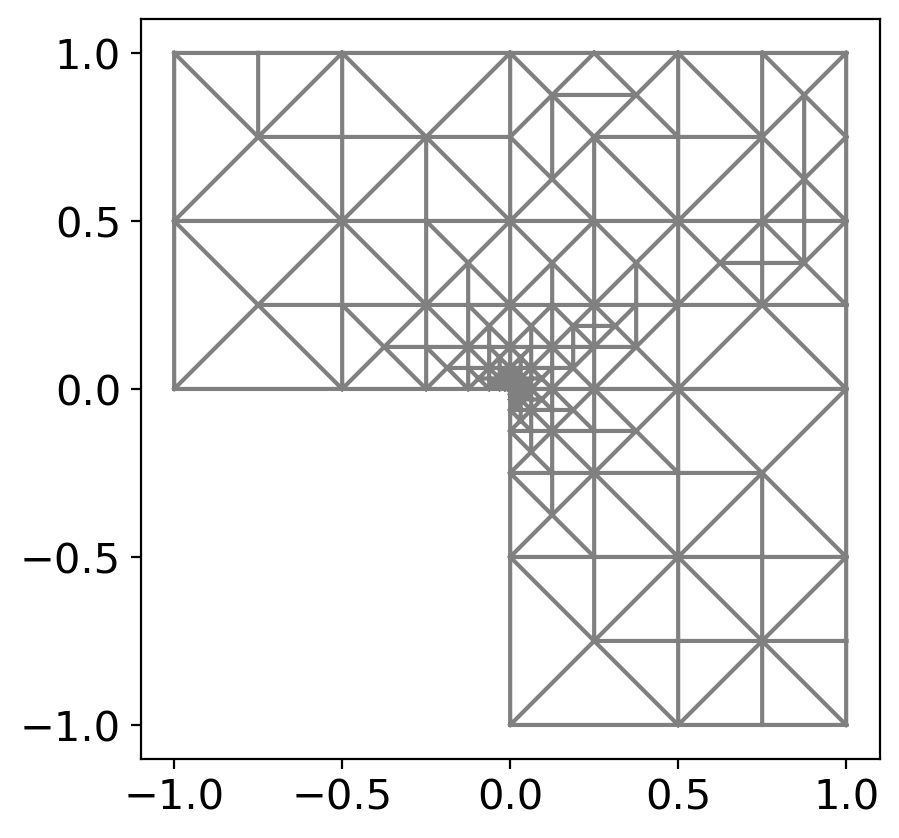}}~
    \subfloat[Pointwise absolute error]{\includegraphics[width=0.39\linewidth]{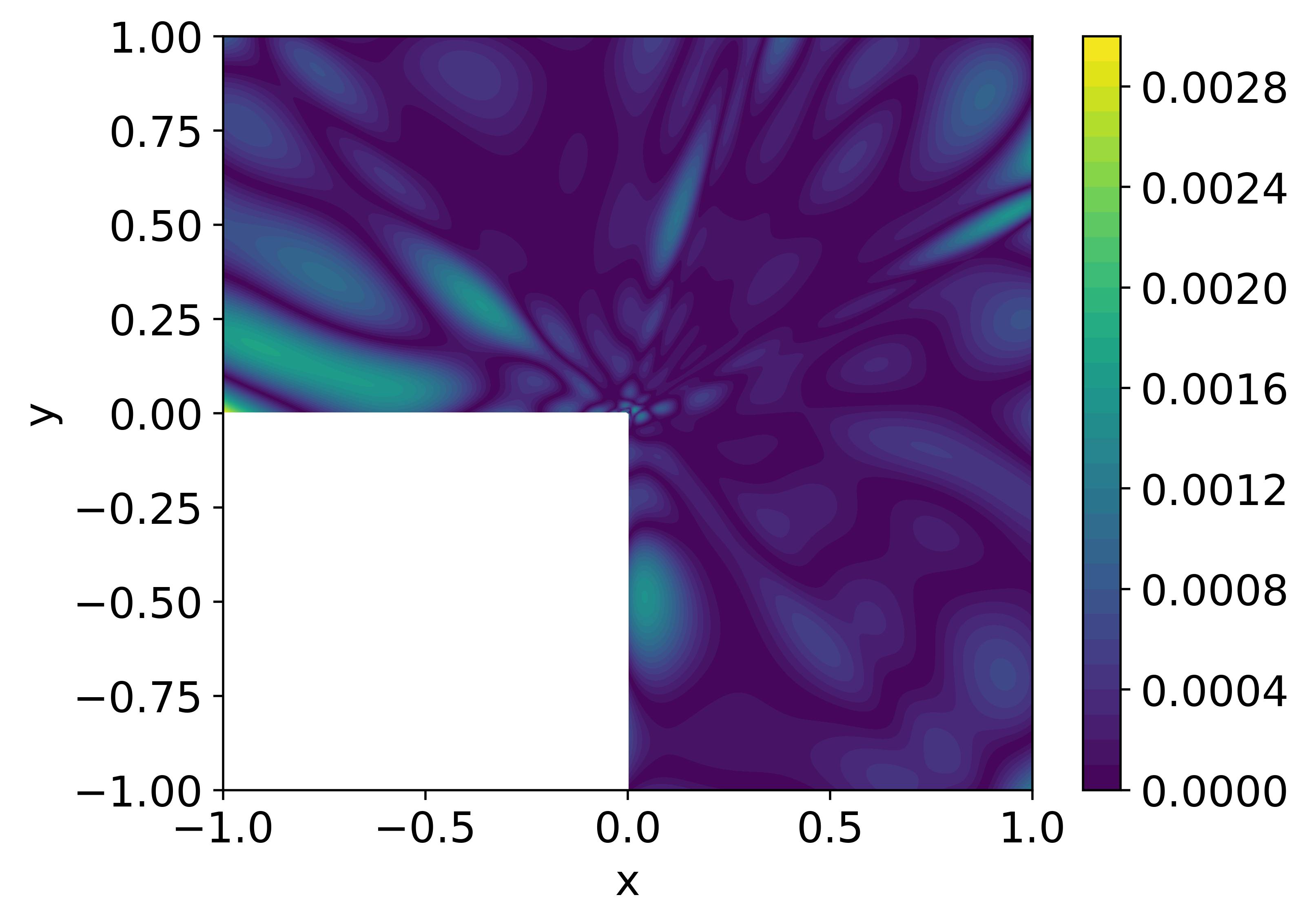}}
    \caption{Test-space meshes and pointwise absolute error for the singular solution.}
    \label{Figure: mesh}
\end{figure}

For the mesh refinement process, we compute two local discrepancy indicators that account for domain and boundary errors. For each element $T$, these indicators are defined as:
    \begin{equation}\label{Eq: two-grid indicator2}
        \iota_{k,\Omega}^2(T;u_\theta) := \|\hat\phi_{k}^\Omega(u_{\theta}) - \phi_{k}^\Omega(u_{\theta})\|_{V_\Omega(T)}^2,
    \end{equation}
and
    \begin{equation}\label{Eq: two-grid indicator3}
        \iota_{k,\Gamma}^2(T;u_\theta) := \|\hat\phi_{k}^\Gamma(u_{\theta}) - \phi_{k}^\Gamma(u_{\theta})\|_{V_\Gamma(T)}^2,
    \end{equation}
where $\hat\phi_k^\Omega(u_{\theta})$ and $\hat\phi_k^\Gamma(u_{\theta})$ are the discrete Riesz representatives with respect to the enriched test spaces $\hat{V}_k^\Omega \supset V_k^\Omega$ and $\hat{V}_k^\Gamma \supset V_k^\Gamma$, respectively. The enriched test spaces are chosen as $\hat{V}_k^\Omega:=\mathbb{P}_0^2(\mathcal{T}_k)$ and $\hat{V}_k^\Gamma := \mathbf{RT}^2(\mathcal{T}_k)$. The global estimator is defined by:
\begin{equation}\label{Eq: global estimator2}
    \iota_k(\mathcal{T}_k;u_\theta):=\sqrt{\sum_{T \in \mathcal{T}_k}\left(\iota_{k,\Omega}^2(T;u_\theta) + \iota_{k,\Gamma}^2(T;u_\theta)\right)}.
\end{equation}

We select elements for refinement using the D\"orfler marking criterion with $\gamma = 0.2$. We follow a separate marking strategy for the domain $\Omega$ and the boundary $\Gamma$. Let $\mathcal{M}_{k-1}^\Omega$ and $\mathcal{M}_{k-1}^\Gamma$ be the smallest subsets of marked elements using indicators $\iota_{k-1,\Omega}^2(T;u_{\theta^{k-1}})$ and $\iota_{k-1,\Gamma}^2(T;u_{\theta^{k-1}})$ from \eqref{Eq: two-grid indicator2} and \eqref{Eq: two-grid indicator3}, respectively. To balance refinement between the two regions, we define $m = \min(|\mathcal{M}_{k-1}^\Omega|,|\mathcal{M}_{k-1}^\Gamma|)$ and select the $m$ elements from each subset with the largest local discrepancies. These selected elements are then combined to form the final set of marked elements $\mathcal{M}_{k-1}$.

Figure \ref{Figure: mesh} (b) illustrates the refined mesh, consisting of $324$ elements. 
As expected, the adaptive refinement is concentrated around the singularity at the origin $(0,0)$. The pointwise absolute errors are shown in Figure \ref{Figure: mesh} (c). It is evident that the neural network approximates the solution with high precision, maintaining absolute errors below $3\times10^{-3}$ across the domain. Figure \ref{fig:error loss adaptive dof} (a) shows the relationship between the square root of the loss and the $H^1$-norm error.
The two curves remain nearly parallel throughout the training process, demonstrating the robustness of the loss. However, after about $10{,}000$ epochs, they begin to diverge slightly. 
This discrepancy arises from an integration error, since refinement is applied only to the test-space mesh. One possible approach to overcome this issue is to use separate meshes: one for the trial space to ensure more accurate integration, and another for the test space dedicated to adaptive refinement.

\begin{figure}[t]
    \centering
    \subfloat[$H^1$-error and $\sqrt{\mathrm{loss}}$ vs. epochs]{\includegraphics[width=0.31\linewidth]{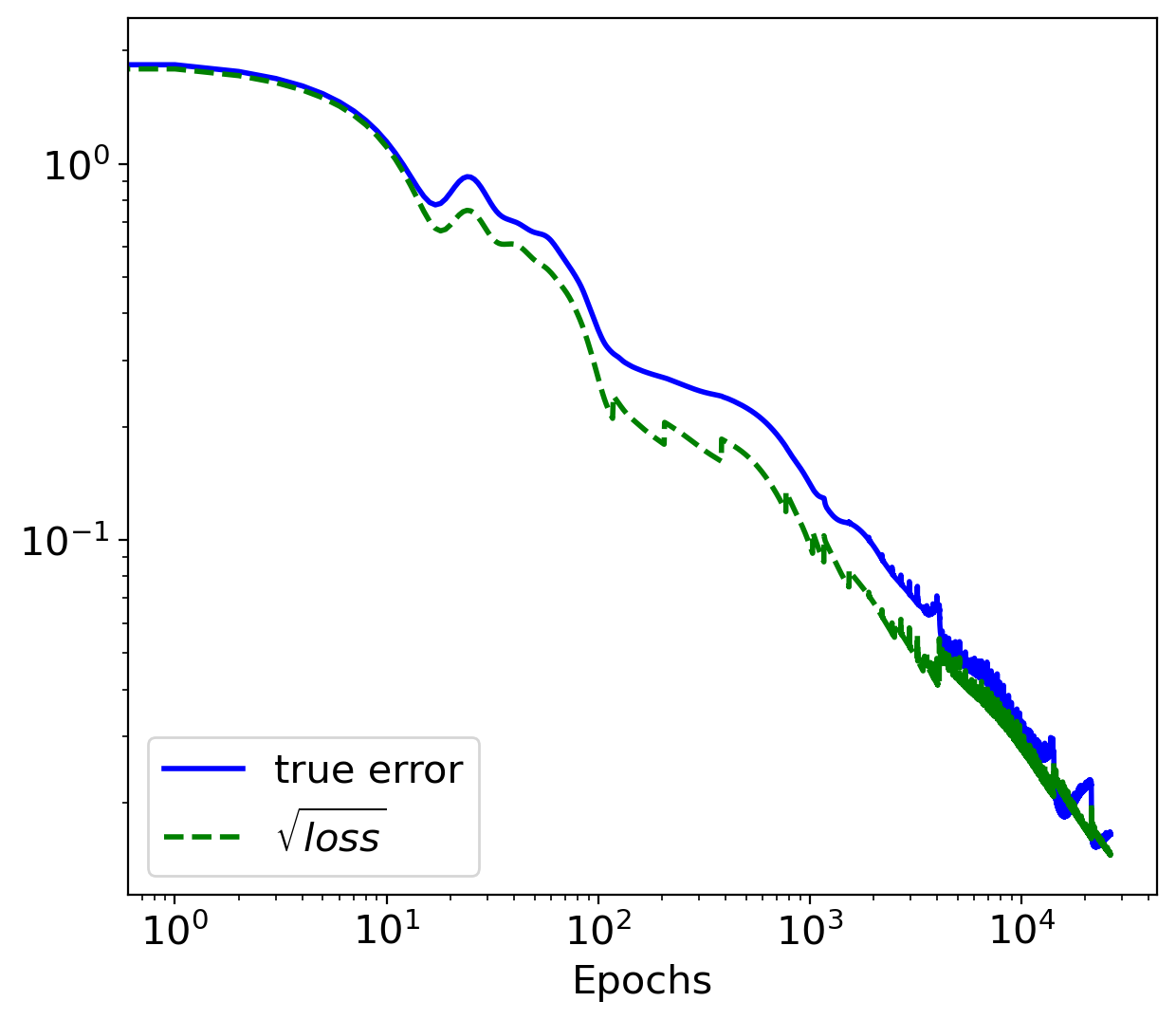}}~
    \subfloat[$H^1$-error, $\sqrt{\mathrm{loss}}$, and refinement indicator vs. test space dimension]{\includegraphics[width=0.31\linewidth]{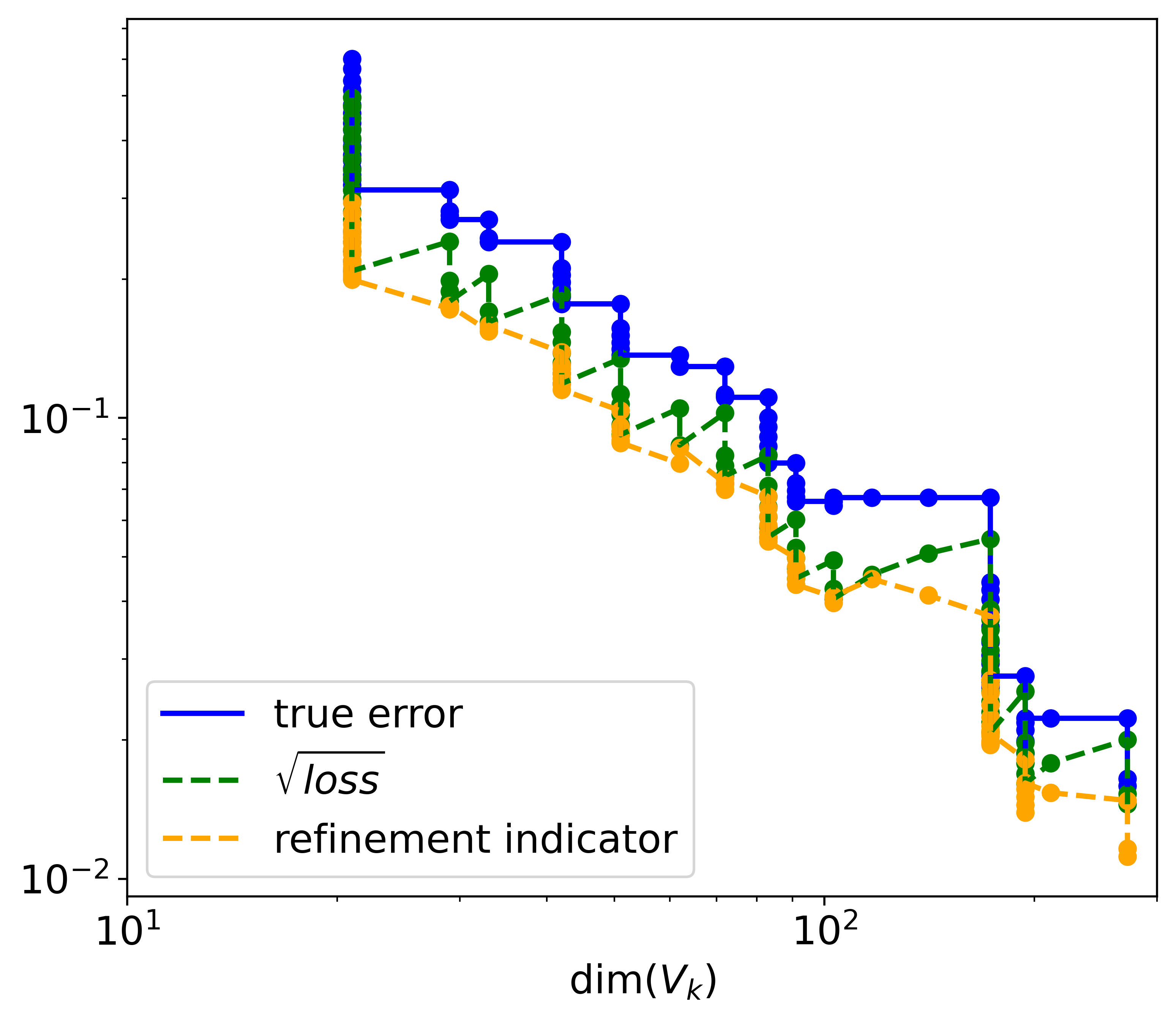}}~ 
    \subfloat[$H^1$-error for FEM and adaptive RVPINNs vs. test space dimension]{\includegraphics[width=0.32\linewidth]{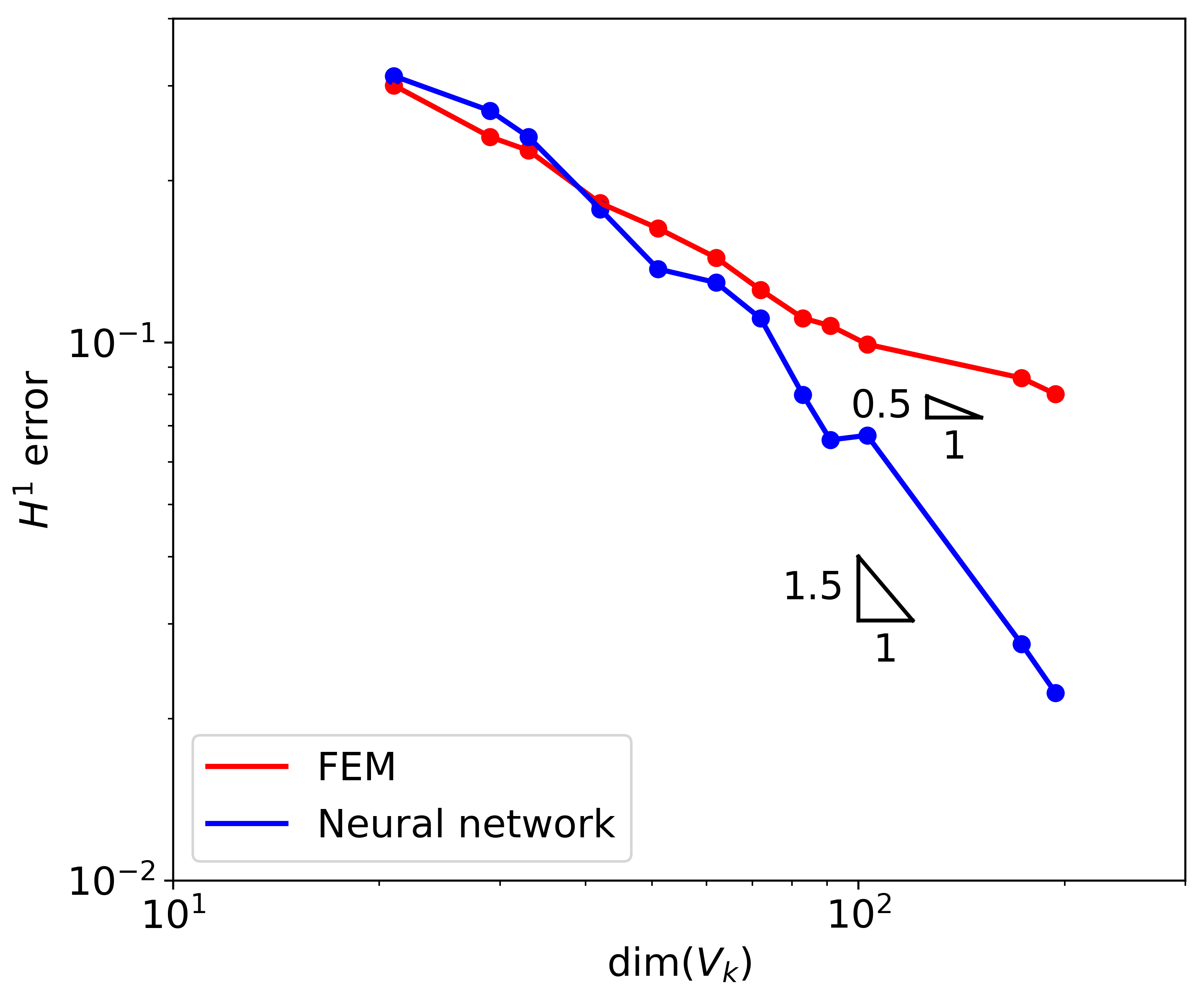}}
    \caption{Results for the singular solution.}
    \label{fig:error loss adaptive dof}
\end{figure}

Figure \ref{fig:error loss adaptive dof} (b) shows the behavior of the square root of the loss function and the refinement indicator as they reach the specified tolerances $\delta^k\epsilon_0$. Consistent with the observations for the smooth solution (Figure \ref{Figure: square Poisson} (c)), the mesh refinement initially increases the loss value, while simultaneously decreasing the refinement indicator. In contrast, neural network training minimizes the loss and generally reduces the indicator. Empirical results confirm that utilizing Algorithm \ref{Algorithm: Adaptive RVPINNs} yields a reliable reduction in the true error.

\pagebreak

Figure \ref{fig:error loss adaptive dof} (c) displays the $H^1$-norm error for both the adaptive RVPINN method and the FEM against the dimension of the discrete test space. The results demonstrate that adaptive RVPINNs significantly improve approximation accuracy. Notably, when $\dim(V_k) > 40$, the $H^1$-norm error of the adaptive RVPINNs is smaller than that of the FEM. Moreover, the adaptive RVPINNs exhibit a convergence rate of approximately $1.5$, which outperforms the $0.5$ rate of the adaptive FEM with continuous piecewise-linear trial and test functions. This superior performance is due to the high expressivity of the neural network architecture and the effectiveness of the refinement indicator.

\section{Conclusions}\label{Sec:Conc}

In this work, we established an upper bound for neural network approximation errors based on Riesz discrepancy and proved its equivalence to the bound based on the residual remainder term. We identified the error between the discrete and continuous Riesz representatives as a key component to ensure robustness during training with the RVPINN loss function. We introduced idealized and sequential adaptivity strategies, and established their convergence results. Furthermore, we proved an error estimate for a quasi-minimizer of the residual norm. Our analysis demonstrated that convergence depends on three primary pillars: the Riesz discrepancy, the neural network's approximability, and the nonlinear optimizer's efficiency in minimizing the loss function. 
We introduced a computable refinement indicator to minimize the Riesz discrepancy. This led to the development of the test-space adaptive RVPINN framework, which refines the test-space mesh to enhance the accuracy of the discrete Riesz representative in approximating its exact counterpart associated with the variational residual. 

We validated the effectiveness of this framework using Poisson problems with smooth and singular solutions, as well as an elliptic interface problem with a kink solution. Numerical results demonstrated that the proposed adaptive approach achieves smaller errors than the reference adaptive FEM, as shown in the plots against the test-space dimension.
Moreover, our approach achieves a better convergence rate with respect to the test-space dimension compared to the adaptive FEM. Furthermore, the results showed that the proposed method maintains a correlation between the loss function and the true error throughout training, demonstrating the robustness of the adaptive strategy.

Potential extensions of this work include developing adaptive strategies that extend the idealized setting by using, for instance, a truncated inner loop ($N<\infty$) to guarantee convergence. In addition, since numerical examples demonstrate that the adaptive process can yield convergence rates superior to those of adaptive FEM, further investigation into the optimal rates for neural network approximations would provide a more comprehensive understanding of the method.



\section*{Acknowledgments}
T. Udomworarat was supported by the Royal Thai Government Scholarship under the Office of Educational Affairs, the Royal Thai Embassy in London. The research by I. Brevis and K. G. van der Zee was supported by the Engineering and Physical Sciences Research Council (EPSRC), UK, under Grant EP/W010011/1. The work of S. Rojas was supported by ANID-FONDECYT project 1240643, and by the National Center for Artificial Intelligence CENIA FB210017, Basal ANID.

\bibliographystyle{siamplain}
\bibliography{references}
\end{document}